\documentclass[11pt, leqno]{article}

\usepackage{amsmath,amssymb,amsthm, epsfig, bm}

\usepackage{hyperref}

\usepackage{amsmath}

\usepackage[pdftex]{color}

\usepackage{mathtools}
\usepackage{color}

\usepackage{ulem}

\usepackage{comment}

\usepackage{mathrsfs}
\usepackage[mathscr]{euscript}

\usepackage{wasysym}

\usepackage{stackrel}

\usepackage{pgfplots}
\pgfplotsset{compat=newest}

\title{\large{\bf Comparison Principle, A.B.P.-type estimates for solutions of quasi-linear elliptic equations in non-divergence form and some implications}}
\author{\it by \smallskip \\ Junior da S. Bessa \footnote{\noindent Universidade Estadual de Campinas - UNICAMP. Instituto de Matem\'{a}tica, Estat\'{i}stica e Computa\c{c}\~{a}o Cient\'{i}fica - IMECC. Departamento  de Matemática. Bar\~{a}o Geraldo, Campinas - SP, Brazil. \noindent \texttt{E-mail address: \url{jbessa@unicamp.br}}},\qquad
Reshmi Biswas\footnote{\noindent  Universidad T\'{e}cnica Federico Santa Mar\'{i}a, Departamento de Matem\'{a}tica, Avda. Espa\~{n}a, 1680, Valparaíso, Chile. \noindent \texttt{E-mail address: \url{rbiswas@usm.cl}}},  \quad Jo\~{a}o Vitor  da Silva
\footnote{\noindent Universidade Estadual de Campinas - UNICAMP Instituto de Matem\'{a}tica, Estat\'{i}stica e Computa\c{c}\~{a}o Cient\'{i}fica - IMECC. Departamento  de Matemática. Bar\~{a}o Geraldo, Campinas - SP, Brazil. \noindent \texttt{E-mail address: \url{jdasilva@unicamp.br}}}, \quad Ginaldo S\'{a}\footnote{\noindent Centro de Modelamiento Matemático (CNRS IRL2807), Universidad de Chile, Santiago, Chile.
\noindent \texttt{E-mail address: \url{gdesantana@dim.uchile.cl}}, 
} \\ \quad $\&$ \\ \quad Makson Santos\footnote{\noindent University of Lisbon, Center for Mathematical Studies (CEMS.UL), Lisboa, Portugal. \noindent \texttt{E-mail address: \url{msasantos@fc.ul.pt}}}
}

\newlength{\hchng}
\newlength{\vchng}
\newcommand{\defeq}{\mathrel{\mathop:}=}

\newtheorem{theorem}{Theorem}[section]

\newtheorem{lemma}[theorem]{Lemma}

\newtheorem{proposition}[theorem]{Proposition}

\newtheorem{corollary}[theorem]{Corollary}

\theoremstyle{definition}

\newtheorem{definition}[theorem]{Definition}

\newtheorem{example}[theorem]{Example}

\theoremstyle{remark}

\newtheorem{remark}[theorem]{Remark}

\numberwithin{equation}{section}

\newcommand{\intav}[1]{\mathchoice {\mathop{\vrule width 6pt height 3 pt depth  -2.5pt
\kern -8pt \intop}\nolimits_{\kern -6pt#1}} {\mathop{\vrule width
5pt height 3  pt depth -2.6pt \kern -6pt \intop}\nolimits_{#1}}
{\mathop{\vrule width 5pt height 3 pt depth -2.6pt \kern -6pt
\intop}\nolimits_{#1}} {\mathop{\vrule width 5pt height 3 pt depth
-2.6pt \kern -6pt \intop}\nolimits_{#1}}}

\begin{document}
\maketitle

\begin{abstract}
\noindent  We establish existence and $L^{\infty}$ bounds--an Alexandroff--Bakelman--Pucci (A.B.P. for short) type maximum principle--for solutions of a class of quasilinear elliptic equations in non-divergence form with a general degeneracy law and a Hamiltonian term, given by
\[
-\Psi(x, |Du|)\Delta_p^{\mathrm{N}}u(x)+\mathscr{H}(x,Du) = f(x) \quad \text{in} \quad \Omega, \qquad 1<p\leq \infty,
\]
under suitable structural assumptions on the data. Our method exploits tools from convex analysis and differential geometry, unveiling the intrinsic geometry of quasilinear operators and their solutions. In particular, our findings are also striking for the $(3-h)$-homogeneous infinity-Laplacian ($h \in [0, 2]$), 
$\Delta^{h}_{\infty} u   =  \frac{1}{|Du|^h} \langle D^2u\, Du, Du \rangle$,
a highly degenerate, fully nonlinear operator closely linked to Tug-of-War games. Furthermore, we derive versions of the Comparison Principle, the Strong Maximum Principle, the Hopf Lemma, and Liouville-type theorems under appropriate assumptions, each carrying independent mathematical significance.


\vspace{0.2cm}
	
\end{abstract}

\newpage

\section{Introduction}\label{Section1}

Recent advances in nonlinear elliptic PDEs have raised several natural questions that remain open, leaving key aspects of the theory incomplete. For motivation, consider the model equation
\begin{equation}\label{Problem}
|Du|^{\theta}\Delta_{p}^{\mathrm{N}}u 
+ \langle \mathfrak{B}(x),Du\rangle |Du|^{\theta} 
+ \varrho(x)|Du|^{\sigma} 
= f(x) \quad \text{in } \mathrm{B}_1 \quad (\text{for}\,\,\, p \in (1, \infty)),
\end{equation}
where $\Delta_{p}^{\mathrm{N}}u$ is the Game-theoretic p-Laplace operator, subject to the continuous Dirichlet boundary condition $u=g$, with $\theta>0$ and $\theta<\sigma<\theta+1$. We highlight the following issues:

\begin{itemize}
    \item[\checkmark] \textbf{Existence and uniqueness:} If the data $\mathfrak{B}, \varrho$, and $f$ are continuous, does \eqref{Problem} admit a viscosity solution, and is it unique?
    
    \textbf{Regularity:} Bessa and Da~Silva proved in \cite[Theorem~1.4]{BessaDaSilva2025} the estimate
    \[
        \|u\|_{\mathrm{C}^{1,\alpha}(\mathrm{B}_{1/2})} 
        \leq \mathrm{C} \cdot\left( 
            \|u\|_{L^{\infty}(\mathrm{B}_{1})} 
            + \|\varrho\|^{\tfrac{1}{1+\theta-\sigma}}_{L^{\infty}(\mathrm{B}_{1})} 
            + \|f\|^{\tfrac{1}{1+\theta}}_{L^{\infty}(\mathrm{B}_{1})} 
        \right),
    \]
    where $\mathrm{C}>0$ depends only on $n$, $p$, $\sigma$, $\theta$, and 
    $\|\mathfrak{B}\|_{L^{\infty}(\mathrm{B}_1;\mathbb{R}^n)}$.
    
    \item[\checkmark] \textbf{$L^{\infty}$ bounds:} Can one bound $\|u\|_{L^{\infty}(\mathrm{B}_{1})}$ solely in terms of the data in \eqref{Problem}, in the spirit of an A.B.P.-type estimate? Such bounds are crucial in several contexts of mathematical and numerical analysis.
    
     \textbf{Qualitative results:} Bessa and Da~Silva obtained a version of the Strong Maximum Principle and the Hopf Lemma in \cite[Theorem~1.9]{BessaDaSilva2025} under appropriate structural conditions.
    
    \item[\checkmark] \textbf{Extensions:} Do the above existence, A.B.P.-type bounds, and qualitative properties extend to broader classes of degenerate or singular quasilinear elliptic equations in non-divergence form with suitable Hamiltonian terms?
    
  \textbf{Key applications:} These estimates and qualitative results have recently been employed to derive sharp geometric bounds in various free boundary problems, including dead-core \cite{alcantara2025}, one-phase Bernoulli-type \cite{BJrDaSRic, daSRRV2023}, and obstacle-type models \cite{DaSilvaVivas2021}, among others. Hence, such results will play a pivotal role in addressing quasilinear equations in non-divergence form across broad settings. 
\end{itemize}

In this work, we establish the existence and Alexandroff-Bakelman-Pucci (A.B.P. for short) type estimates for both singular and degenerate quasilinear equations governed by the normalized \(p\)-Laplacian operator, given by
\begin{eqnarray}\label{1.1}
-\Psi(x,|D u|)\Delta_{p}^{\mathrm{N}}u+\mathscr{H}(x, D u)=f(x)\,\,\, \mbox{in}\,\,\, \Omega,
\end{eqnarray}
where $(x, \xi) \mapsto \Psi(x, |\xi|), \mathscr{H}(x, \xi)$ are suitable functions (to be specified \textit{a posteriori}), $f \in L^\infty(\Omega) \cap C^0(\Omega)$, $\Omega$ is a domain in $\mathbb R^n$, and \(1<p\leq +\infty\). Furthermore, we address versions of the Comparison Principle, the Strong Maximum Principle, the Hopf Lemma, and Liouville-type results under appropriate assumptions on data, to be detailed later.

For a fixed \( p \in (1, \infty) \), the operator
\[
\Delta^{\mathrm{N}}_p u = |D u|^{2-p} \operatorname{div}(|D u|^{p-2} D u) = \Delta u + (p-2) \left\langle D^2 u \frac{D u}{|D u|}, \frac{D u}{|D u|}\right\rangle
\]
denotes the \textit{Normalized \( p \)-Laplacian operator} or \textit{Game-theoretic \( p \)-Laplace operator}, which arises in certain stochastic models (see \cite{BlancRossi-Book}, \cite{Lewicka20}, and \cite{Parvia2024} for comprehensive surveys on Tug-of-War games and their associated PDEs). The fully nonlinear degenerate operator
\[
\Delta_{\infty}^{\mathrm{N}} u(x) \coloneqq \left\langle D^2 u \frac{D u}{|D u|}, \frac{D u}{|D u|}\right\rangle
\]
is referred to as the \textit{Normalized infinity-Laplacian}, which also plays a significant role in the interplay between nonlinear PDEs and Tug-of-War games (cf. \cite{BlancRossi-Book} and \cite{Parvia2024}).
\medskip

The function $\Psi(\cdot, \cdot)$ plays a central role in this article, as it determines whether the equation is degenerate or singular. Moreover, $\Psi$ allows us to recover several well-known degeneracy laws from the literature, for instance, the powers of the gradient when $\Psi(x,t) := t^{\hat p}$, as explained in the following. 

Now, recall that a function \(g:(0,+\infty) \to\mathbb{R}\) is said \textit{almost non-decreasing} with constant \(\mathfrak{L} \geq 1\) if
\[
g(s) \leq \mathfrak{L}g(t), \text{ for any } 0 < s \leq t.
\]
The definition of \textit{almost non-increasing} function with constant \(\mathfrak{L} \geq 1\) is similar.

Throughout this manuscript, we impose the following assumptions on $\Psi$, which are general enough to encompass many important cases studied in the literature (cf. \cite{BBLL24}). 

\begin{itemize}
\item [\((\Psi1)\)] For any domain $\Omega \subset \mathbb R^n$, the vector field $\Psi: \Omega\times [0,+\infty)\to[0,+\infty)$ is a continuous function such that there exist positive constants \(0<\mathfrak{a}\leq \mathfrak{b}\) satisfying
\[
\mathfrak{a}\leq \Psi(x,1)\leq \mathfrak{b},\, \text { for all } x\in \Omega.
\]
\item[\((\Psi2)\)] There exist constants \(s_{\Psi}\geq i_{\Psi}>-1\) such that the map \(t\longmapsto \frac{\Psi(t)}{t^{i_{\Psi}}}\) is almost non-decreasing with constant \(\mathfrak{L}_{1} \geq 1\) and the map \(t\longmapsto\frac{\Psi(t)}{t^{s_{\Psi}}}\) is almost non-increasing with constant \(\mathfrak{L}_{2} \geq 1\) in \((0,+\infty)\), for all \(x\in \Omega\).
\end{itemize}

On the other hand, the Hamiltonian $\mathscr H$ will satisfy the following conditions (unless mentioned otherwise - cf. \cite{BessaDaSilva2025}):
\begin{itemize}
\item[$(\mathbf{H1})$] We also suppose that the 
Hamiltonian term \(\mathscr{H}:\Omega\times\mathbb{R}^{n}\to\mathbb{R}\) satisfies the following condition
\[
|\mathscr{H}(x,\xi)|\leq \varrho(x)|\xi|^{\sigma},
\]
for all \((x,\xi)\in \Omega\times \mathbb{R}^{n}\), for some positive function \(\varrho\in L_+^{\infty}(\Omega)\cap C^{0}(\Omega)\) and $0\leq \mathrm{c}_0 <\sigma < \mathrm{c_1}< \infty$.


\end{itemize}
\begin{remark}\label{obsdePsi}
By conditions $(\Psi1)$ and $(\Psi2)$, it is possible to verify that:
\begin{itemize}
\item [$i)$] For any \(x\in \Omega\) and \(0<t\leq 1\), it holds that
\[
\Psi(x,t)\geq \mathfrak{L}_{2}\mathfrak{a}t^{s_{\Psi}}.
\]
\item [$ii)$] For any \(x\in \Omega\) and \(t> 1\), it holds that
\[
\Psi(x,t)\geq \frac{\mathfrak{a}}{\mathfrak{L}_{1}}t^{i_{\Psi}}.
\]
\end{itemize}
\end{remark}

\begin{example} 
In the sequel, we present some examples of vector fields  $\Psi$, which satisfy the above assumptions:

\begin{enumerate}
    \item {\bf Constant exponent power-type:}  
    \[
    \Psi(x, t) = |t|^{\hat{p}} \quad \text{for } \hat{p} > -1, \quad \text{in this context, we have} 
    \quad i_\Psi = s_\Psi = \hat{p}.
    \]

    \item {\bf Double phase-type:}  
    \[
    \Psi(x, t) = |t|^{\hat{p}} + \mathfrak{a}(x)|t|^{\hat{q}} \quad \text{for } -1 < \hat{p} < \hat{q} < \infty \text{ and } 0 \leq \mathfrak{a}(x) \in C^0(\Omega), 
    \]
    here, we have
    $i_\Psi = \hat{p}$, and $s_\Psi = \hat{q}$.
\item {\bf Borderline case of double phase-type:}  
    \[
    \Psi(x, t) = |t|^{\hat{p}} + \mathfrak{a}(x)|t|^{\hat{p}} \log(1+t) \quad \text{for } -1 < \hat{p} < \infty \text{ and } 0 \leq \mathfrak{a} \in C^0(\Omega),
    \]
here, we have $i_\Psi =\hat{p}$, and $s_\Psi = \hat{p}+1$.
    \item {\bf Variable exponent power-type:}  
    \[
    \Psi(x, t) = |t|^{\hat{p}(x)} \quad \text{for } \hat{p} \in C^0(\Omega) \text{ and } -1 < \inf_{\Omega} \hat{p}(x) \leq \sup_{\Omega} \hat{p}(x) < \infty.\]
    In such a context, we have $\displaystyle i_\Psi = \inf_{\Omega} \hat{p}(x), \quad \text{and} \quad  s_\Psi = \sup_{\Omega} \hat{p}(x).$

    \item {\bf Variable exponent double phase-type:}  
    \[
    \Psi(x, t) = |t|^{\hat{p}(x)} + \mathfrak{a}(x)|t|^{\hat{q}(x)} \quad \text{for } \hat p, \hat q \in C^0(\Omega),
    \]
    \[
    -1 < \inf_{\Omega} \min\{\hat{p}(x), \hat{q}(x)\}, \quad \sup_{ \Omega} \max\{\hat{p}(x), \hat{q}(x)\} < \infty, \quad \text{and } 0 \leq \mathfrak{a} \in C^0(\Omega):
    \]
     In such a context, we have $\displaystyle     i_\Psi = \inf_{ \Omega} \min\{\hat{p}(x), \hat{q}(x)\},  \quad \text{and} \quad s_\Psi = \sup_{\Omega} \max\{\hat{p}(x), \hat{q}(x)\}.$

     \item {\bf Borderline case of variable double phase-type:}  
    \[
    \Psi(x, t) = |t|^{\hat{p}(x)} + \mathfrak{a}(x)|t|^{\hat{p}(x)} \log(1+t).
    \]
for $\hat{p} \in C^0(\Omega)$ with and $\displaystyle -1 < \inf_{\Omega} \hat{p}(x) \leq \sup_{\Omega} \hat{p}(x) < \infty$, and $0 \leq \mathfrak{a} \in C^0(\Omega)$. In this scenario, we have $\displaystyle i_\Psi = \inf_{\Omega} \hat{p}(x)$, and $ \displaystyle s_\Psi = \sup_{\Omega} \hat{p}(x)+1$.
\end{enumerate}

\end{example}

\begin{example}  
Typical  examples of Hamiltonians $\mathscr{H} :\Omega\times\mathbb{R}^{n}\to\mathbb{R}$ that satisfy the condition \(({\bf H1})\) include:
\begin{itemize}
\item [$i)$] \(\mathscr{H}(x,{\xi})=\langle \mathfrak{B}(x),{\xi}\rangle|{\xi}|^{i_\Psi}+\varrho(x)|{\xi}|^{\sigma}\), for $0<\mathrm{c}_0< \sigma<\mathrm{c}_1=1+i_{\Psi}$, where the vector field \(\mathfrak{B}\) and \(\varrho >0\) belong to $L^{\infty}(\Omega)\cap C^{0}(\Omega)$. 

\item[$ii)$]  \(\mathscr{H}(x,{\xi})=\mathfrak{a}(x)|{\xi}|^{\theta}+\mathfrak{b}(x)|{\xi}|^{\sigma}\), for $0<\mathrm{c}_0 = \theta< \sigma < \mathrm{c}_1 = 1+i_{\Psi} $, where $\mathfrak{a}, \mathfrak{b} \in L^{\infty}(\Omega)\cap C^{0}(\Omega)$.
\item[$iii)$] In the previous examples $i), ii)$, replace $i_\Psi$ by $s_\Psi$.

\end{itemize}
\end{example}
\medskip
Now we are ready to present the first main result of this article, Alexandroff-Bakelman-Pucci (A.B.P.) Maximum Principle, for the case $p=\infty$.

\begin{theorem}\label{Theorem3.1}
Let $\Omega\subset \mathbb R^n$ be a bounded domain and  \(u\in C^{0}(\overline{\Omega})\) be a viscosity solution to 
\begin{align}\label{inftt}
-\Psi(x,|D  u|)\Delta_{\infty}^{\mathrm{N}}u+\mathscr{H}(x,D  u)\leq f(x)\,\,\, \mathrm{in}\,\,\, \Omega. 
\end{align}
such that $\Psi: \Omega\times [0, \infty) \to [0, \infty)$ satisfies the structural assumptions $(\Psi1), (\Psi2)$, the  Hamiltonian term \(\mathscr{H}:\Omega\times\mathbb{R}^{n}\to\mathbb{R}\) satisfies $(\mathbf{H1})$ with  $\mathrm{c}_0>0$ and  $\mathrm{c}_1=2+i_\Psi$, and $f\in C^{0}(\Omega) \cap L^{\infty}(\Omega)$. Then, 
\begin{eqnarray*}
\min\left\{\left(\frac{\sup\limits_{\Omega}u-\sup\limits_{\partial \Omega}u^{+}}{\operatorname{diam}(\Omega)}\right)^{2+i_{\Psi}-\sigma},\left(\frac{\sup\limits_{\Omega}u-\sup\limits_{\partial \Omega}u^{+}}{\operatorname{diam}(\Omega)}\right)^{2+s_{\Psi}}\right\}\\\leq\mathrm{C}\int\limits_{\sup\limits_{\partial\Omega} u^{+}}^{\sup\limits_{\Omega}u}\|(f^{+}+\varrho)\chi_{\mathrm{C}(u^{+})}\|_{L^{\infty}(\{u^{+}=\tau\})}d\tau.
\end{eqnarray*}
Here,  we have extended \(u^{+}\) by zero outside \(\Omega\), \(\mathrm{C}(u^{+})\) is the contact set with the concave envelope of \(u^{+}\) in the convex hull  \(\Omega^{*}=\operatorname{conv}(\Omega)\) and \(\mathrm{C}=\frac{(2+s_{\Psi})\mathfrak{L}_1}{a}\).

Similarity, if \(u\) satisfies 
\[
-\Psi(x,|D  u|)\Delta_{\infty}^{\mathrm{N}}u+\mathscr{H}(x,D  u)\geq f(x)\,\,\, \mbox{in}\,\,\, \Omega
\]
in the viscosity sense, then
\begin{eqnarray*}
\min\left\{\left(\frac{\sup\limits_{\Omega}u^{-}-\sup\limits_{\partial \Omega}u^{-}}{\operatorname{diam}(\Omega)}\right)^{2+i_{\Psi}-\sigma},\left(\frac{\sup\limits_{\Omega}u^{-}-\sup\limits_{\partial \Omega}u^{-}}{\operatorname{diam}(\Omega)}\right)^{2+s_{\Psi}}\right\}\\
\leq\mathrm{C}\int\limits_{\sup\limits_{\partial\Omega} u^{-}}^{\sup\limits_{\Omega}u^{-}}\|(f^{-}+\varrho)\chi_{\mathrm{C}(u^{-})}\|_{L^{\infty}(\{u^{-}=\tau\})}d\tau.
\end{eqnarray*}
\end{theorem}

\begin{remark}
In the Theorem \ref{Theorem3.1}, when \(\Psi(x,t)=1\) and \(\mathscr{H}(x,\xi)=0\), the A.B.P.-type estimates for subsolutions of the normalized infinity Laplacian obtained in \cite{Charro13} are recovered, since $i_{\Psi} = s_{\Psi} = 0$ and $\mathfrak{L}_{1} = \mathfrak{L}_{2} = \mathfrak{a} = \mathfrak{b} = 1$.
\end{remark}

\medskip
The A.B.P.-type estimates obtained in Theorem \ref{Theorem3.1} allow us to establish analogous estimates for the normalized \(p\)-Laplacian for \(p \in (1,\infty)\). More precisely, we have the following result: 

\begin{theorem}\label{A.B.P.caseforp}
Let $\Omega\subset \mathbb R^n$ be a bounded domain and \(u\in C^{0}(\overline{\Omega})\) be a viscosity solution to
\begin{align}\label{plll}
-\Psi(x,|D  u|)\Delta_{p}^{\mathrm{N}} u + \mathscr{H}(x,D  u)\leq f(x) \quad \text{in} \quad \Omega,
\end{align}
such that $\Psi: \Omega\times [0, \infty) \to [0, \infty)$ satisfies the structural assumptions $(\Psi1), (\Psi2)$, the  Hamiltonian term \(\mathscr{H}:\Omega\times\mathbb{R}^{n}\to\mathbb{R}\) satisfies $(\mathbf{H1})$ with  $\mathrm{c}_0>0$ and  $\mathrm{c}_1=2+i_\Psi$, and $f\in C^{0}(\Omega) \cap L^{\infty}(\Omega)$. Then,
\begin{align}\label{est1A.B.P.}
\min\left\{\left(\frac{\sup\limits_{\Omega}u-\sup\limits_{\partial \Omega}u^{+}}{\operatorname{diam}(\Omega)}\right)^{2+i_{\Psi}-\sigma},\left(\frac{\sup\limits_{\Omega}u-\sup\limits_{\partial \Omega}u^{+}}{\operatorname{diam}(\Omega)}\right)^{2+s_{\Psi}}\right\}\nonumber\\
\leq\mathrm{C}\int\limits_{\sup\limits_{\partial\Omega} u^{+}}^{\sup\limits_{\Omega}u}\|(f^{+}+\varrho)\chi_{\mathrm{C}(u^{+})}\|_{L^{\infty}(\{u^{+}=\tau\})}d\tau,
\end{align}
where $\mathrm{C}=\frac{(2+s_{\Psi})\mathfrak{L}_1p}{(p-1)\mathfrak{a}}$. 
\end{theorem}

\medskip

\subsubsection*{The geometric ideas and insights behind the A.B.P.'s proof}

The geometric approach used to derive the A.B.P. estimate for model \eqref{1.1}, inspired by \cite{Charro13}, relies on analyzing the geometry of the contact sets associated with the solution. In contrast to \cite{BBLL24} and \cite{DFQ2009}, which address related results for fully nonlinear equations with degenerate structures and drift terms, this method combines tools from convex analysis, geometric measure theory, and differential geometry to examine the contact sets in detail. For Theorems \ref{Theorem3.1} and \ref{A.B.P.caseforp}, we proceed as follows:

\begin{itemize}
\item[\checkmark] For \(p=\infty\), we first analyze the case where the solution is semiconvex. This ensures the $C^{1,1}$ regularity of the concave envelope of the positive part of \(u\), namely \(\Gamma(u^+)\) (see Lemma~\ref{lemma5}).

\item[\checkmark] We derive a condition controlling the influence of the gradient magnitude of the concave envelope, accounting for the homogeneity regimes of the vector field \(\Psi\) and the Hamiltonian term \(\mathscr{H}\) through the integral~\eqref{integralI}. Under the regularity of the concave envelope, we establish a relation between its Hessian and the infinity-Laplacian, invoking that \(u\) is a sub/supersolution to~\eqref{inftt}. A suitable version of the Gauss--Bonnet Theorem~\ref{theorem9} then yields the desired estimates for model~\eqref{1.1} when \(p=\infty\).

\item[\checkmark] For the general case, we employ sup-convolutions: from \(u\), we construct a sequence \((u^{\varepsilon})_{\varepsilon>0}\) in subdomains of \(\Omega\) at distance \(\varepsilon\) from the boundary,
\[
u^{\varepsilon}(x)=\sup_{y\in \overline{\Omega}}\left\{u(y)-\frac{1}{2\varepsilon}|x-y|^2\right\}.
\]
These functions are semiconvex and remain subsolutions to~\eqref{inftt}. Applying the previous estimates together with \(\varepsilon\)-stability and letting \(\varepsilon \to 0\) yields Theorem~\ref{Theorem3.1}.

\item[\checkmark] Finally, for Theorem~\ref{A.B.P.caseforp}, we analyze the normalized \(p\)-Laplacian, rewriting it in terms of the infinity-Laplacian and the 1-Laplacian, interpreted as the product of the divergence of the field \(\frac{Du}{|Du|}\) and \(|Du|\). Crucially, we use the geometric identity
\[
\operatorname{div}\left(\frac{D \Gamma_{r}(u^{+})(x)}{|D \Gamma_{r}(u^{+})(x)|}\right)=-\sum_{i=1}^{n-1}\kappa_{i}(x),
\]
where \(\kappa_{i}(x)\) \((i=1,\ldots,n-1)\) denote the principal curvatures of \(\{\Gamma_{r}(u^{+})=r\}\) at \(x\). This approach, combined with the arguments of Theorem~\ref{Theorem3.1}, yields the result with additional refinements controlling he influence of the gradient magnitude of the concave envelope.
\end{itemize}

\medskip
\subsection*{Comparison Principle and some consequences}

In our next theorem,  we present a version of the Comparison Principle for a class of fully nonlinear operators involving a very general class of Hamiltonian terms. Let $G: \Omega \times \mathbb{R}^n \times \text{Sym}(n) \to \mathbb{R}$ be a fully nonlinear  operator given by
\begin{align}\label{g}
\mathcal{G}(x, \xi, \mathrm{X}) = \Psi(x, |\xi|)F_0(\xi, \mathrm{X}) + \mathscr{H}(x, \xi)
\end{align}
where $F_0: \mathbb{R}^n \times \text{Sym}(n) \to \mathbb{R}$ is a fully nonlinear proper operator of ``normalized-type''. For instance, we may consider:
$$
F_0(Du, D^2 u) := \left\{
\begin{array}{ccl}
   \Delta_{\infty}^{\mathrm{N}} u  & = & \frac{1}{|Du|^2} \langle D^2u Du, Du \rangle, \\[0.2cm]
   \Delta^{h}_{\infty} u  & = & \frac{1}{|Du|^h} \langle D^2u Du, Du \rangle, \,\,\ h \in [0, 2], \\[0.2cm]
   \Delta_{p}^{\mathrm{N}} u  & = & \Delta u + (p-2) \Delta^{\mathrm{N}}_{\infty} u, \, 1<p< \infty, \\[0.2cm]
   \Delta_{p(x)}^{\mathrm{N}} u  & = & \Delta u + (p(x)-2) \Delta^{\mathrm{N}}_{\infty} u, \, 1<p^{-}\leq p(x)\leq p^{+}< \infty.
\end{array}
\right.
$$
Here $\Psi : \Omega \times [0,\infty) \to [0,\infty)$ is a continuous map satisfying the properties $(\Psi_1)$ and $(\Psi_2).$
Additionally, $\mathscr{H}: \Omega \times \mathbb{R}^n \to \mathbb{R}$ is a continuous function satisfying:
\begin{align}\label{H}\tag{$\mathbf{H2}$}
 |\mathscr{H}(x,{\xi}) - \mathscr{H}(y,{\xi})| \leq \omega(|x - y|)(1 + \mathscr{H}_0(|{\xi}|)),
\end{align}
 where \( \omega: [0, \infty) \to [0, \infty) \) is a modulus of continuity, i.e., an increasing function such that \( \omega(0) = 0 \), and $\mathscr{H}_0: \mathbb{R}^n \to \mathbb{R}$ is a continuous function.

\medskip

\noindent A general version of the comparison principle for the operator $G$ reads as follows.

\begin{theorem}[\bf Comparison Principle - general version]\label{CPG}
Let $\Omega\subset \mathbb R^n$ be  a bounded domain  and let $\mathfrak{c}, \mathfrak{f}_{1}, \mathfrak{f}_{2} \in \mathrm{C}^0(\overline{\Omega})$, and let $\mathscr{F} : \mathbb{R} \to \mathbb{R}$ be a continuous and increasing function such that $\mathscr{F}(0) = 0$. Suppose that $\mathfrak{u}, \mathfrak{v} \in \mathrm{C}^{0,1}_{\mathrm{loc}}(\Omega)$ are functions satisfying
\begin{equation}\label{CP-General}
\left\{
\begin{array}{rcll}
\mathcal{G}(x, D \mathfrak{u}, D^2 \mathfrak{u}) + \mathfrak{c}(x) \mathscr{F}(\mathfrak{u}) & \geq & \mathfrak{f}_1(x) & \text{in } \Omega, \\[0.2cm]
\mathcal{G}(x, D \mathfrak{v}, D^2\mathfrak{v})  + \mathfrak{c}(x) \mathscr{F}(\mathfrak{v}) & \leq & \mathfrak{f}_2(x) & \text{in } \Omega,
\end{array}
\right.
\end{equation}
in the viscosity sense, where the operator $\mathcal{G}$ is defined in \eqref{g} and $\mathscr{H}$ is a Hamiltonian function satisfying assumption \eqref{H}. Furthermore, assume that $\mathfrak{v} \geq \mathfrak{u}$ on $\partial \Omega$, and that one of the following conditions holds:
\begin{enumerate}
\item[\textnormal{\textbf{(a)}}] $\mathfrak{c} < 0$ in $\overline{\Omega}$ and $\mathfrak{f}_1 \geq \mathfrak{f}_2$ in $\overline{\Omega}$;
\item[\textnormal{\textbf{(b)}}] $\mathfrak{c} \leq 0$ in $\overline{\Omega}$ and $\mathfrak{f}_1 > \mathfrak{f}_2$ in $\overline{\Omega}$.
\end{enumerate}
Then, $\mathfrak{u} \leq \mathfrak{v}$ in $\Omega$.
\end{theorem}

\medskip
{Next, we show a version of the Strong Maximum Principle and the Hopf Lemma, which describe qualitative properties of solutions, by making use of suitable barrier functions}. 

\begin{theorem}[{\bf Strong Maximum Principle}]
  \label{thm:SMP}
Let \( \Omega \subset \mathbb{R}^n \) be a bounded domain, and let  \( \mathfrak{c} \leq 0 \) be a continuous functions on \( \overline{\Omega} \). Suppose that \( v \in \mathrm{C}_{\mathrm{loc}}^{0,1}(\Omega) \) is a non-negative viscosity supersolution of
\begin{equation}
\Psi(x,|D v|)\Delta^{\mathrm{N}}_p v + \mathscr{H}(x,Dv) + \mathfrak{c}(x)v^{1+\sigma}(x) = 0, \quad x \in \Omega,
\end{equation}
where $\Psi$ satisfies the structural assumptions  $(\Psi1), (\Psi2)$ and  the  Hamiltonian term \(\mathscr{H}:\Omega\times\mathbb{R}^{n}\to\mathbb{R}\) satisfies $(\mathbf{H1})$ (with  $\mathrm{c}_0=s_{\Psi}$ and  $\mathrm{c}_1=1+s_\Psi$), and \eqref{H}.
Then either \( v \equiv 0 \) or \( v > 0 \) in \( \Omega \). Furthermore, if \( \Omega \) satisfies the interior sphere condition and \( v(x) > v(z) = 0 \) for all \( x \in \Omega \) and some \( z \in \partial \Omega \), then there exists a constant \( \mathfrak{d} > 0 \) such that
\begin{equation}
v(x) \geq \mathfrak{d}\cdot (r - |x - x_0|), \quad \text{for } x \in B_r(x_0),
\tag{2.7}
\end{equation}
where \( B_r(x_0) \subset \Omega \) is a ball tangent to \( z \).  
\end{theorem}
\medskip
Finally, we present the last result of this article: a Liouville-type theorem that plays a central role in the qualitative analysis of solutions, as it provides a nonexistence criterion under suitable decay assumptions at infinity. This result serves as the counterpart, for non-divergence form quasilinear operators with a general Hamiltonian term, to the one established by Biswas and Vo in \cite[Theorem 2.7]{BisVo23}.
\begin{theorem}[\textbf{Liouville-type result}]\label{LT}
Let  $u$ be a locally Lipschitz viscosity solution to \begin{align}\label{me}
\Psi(x,|Du(x)|)\,\Delta_{p}^{\mathrm{N}}u(x) + \mathscr{H}(x, D u)  \leq 0, \qquad in    \quad \mathbb R^n
\end{align}
 with $\displaystyle\inf_{\mathbb R^n} u > -\infty,$ such that $\Psi: \mathbb{R}^n\times [0, \infty) \to [0, \infty)$ satisfies the structural assumptions $(\Psi1), (\Psi2)$ and the Hamiltonian term \(\mathscr{H}:\mathbb R^n\times\mathbb{R}^{n}\to\mathbb{R}\) is given as\[\mathscr{H}(x, \xi)=\langle\mathfrak{B}(x), \xi\rangle \Psi(x,|\xi|)+ \varrho(x)|\xi|^{\sigma},\] where  $s_\Psi<\sigma <1+s_\Psi$, $\mathfrak B \in C_0^0(\mathbb R^n; \mathbb R^n),$ and $0\leq\varrho \in L^\infty(\mathbb R^n)\cap C^0
(\mathbb R^n)$. Assume that \begin{align}\label{hp}\displaystyle\limsup_{|x|\to +\infty}(\mathfrak B(x)\cdot x)_+ < n.\end{align} Then, $u$  is
necessarily a constant.
\end{theorem}
\medskip

\subsection*{State-of-the-Art and some motivations}


In the history of modern Mathematical Analysis, Maximum Principles have been some of the most useful tools used to solve a wide range of problems in the study of elliptic/parabolic PDEs
over the last century, of which we should cite the Weak, Strong, and Hopf's Maximum Principles as well-known examples of such classical results. Such fundamental tools play a decisive role in the context of divergent and non-divergent PDEs in obtaining uniqueness results, sign restrictions, and a priori estimates for their ``weak solutions'' (cf. \cite{pucciserrin} for an essay on this topic).  To illustrate, we have the following well-known version:  Let $\Omega \subset \mathbb{R}^n$ be a bounded domain and
  $$
\displaystyle  \mathcal{L}u(x) = \sum_{i,j=1}^{n} a_{ij}(x)D_{ij}u + \sum_{i=1}^{n} b_{i}(x).D_iu(x) + c(x)u(x)  =f(x) \quad \text{in} \quad \Omega,
  $$
where $a_{ij} , b_i, c \in L^{\infty}(\Omega)\cap C^0(\Omega)$ satisfying $c \le 0$ in $\Omega$ and
  $a_{ij}(x)\xi_i\xi_j \ge \lambda|\xi|^2$ for a constant $\lambda>0$. Therefore,

\begin{theorem}[{\bf \cite[Theorem 1.1.10]{Han17}}]
  Suppose $u \in C(\overline{\Omega})\cap C^2(\Omega)$ be a solution of
  $$
  \left\{
  \begin{array}{rcrcl}
    \mathcal{L}u(x) & = & f(x) & \text{in}& \Omega, \\
    u(x) & = & \phi(x) & \text{on}& \partial \Omega,
  \end{array}
  \right.
  $$
  for $f \in L^{\infty}(\Omega) \cap C(\Omega)$ and $\phi \in C(\partial \Omega)$. Then,
  $$
  \displaystyle \sup_{\Omega} |u| \le \max_{\partial \Omega} |\phi| + \mathrm{C}(n, \lambda, \|b_{i}\|_{L^{\infty}(\Omega)})\operatorname{diam}(\Omega)^2\sup_{\Omega} |f|.
  $$

\end{theorem}

As for other versions of the so-called classical Maximum Principle, in the mid-1960s, in completely independent works, the mathematicians Alexandrov in \cite{Aleksandrov}, Bakel'man in \cite{Bakelman}, and Pucci in \cite{Pucci}, established a Maximum Principle for elliptic linear equations of $2$nd order in non-divergent form with bounded coefficients. Namely, this result ensures that whenever $u\in C(\overline{\Omega})\cap W^{2,n}_{loc}(\Omega)$ is a viscosity sub-solution (Strong sub-solution) of
\begin{eqnarray}\label{eqA.B.P.}
-\operatorname{Tr}(\mathfrak{A}(x)D^2u(x))+\overrightarrow{b}(x)\cdot D  u(x)+c(x)u(x)= f(x) \quad \text{in} \quad \Omega
\end{eqnarray}
where $\mathfrak{A}(x)$ is a uniformly elliptic matrix, $\overrightarrow{b} \in L^{\infty}(\Omega; \mathbb{R}^n)$, $0 \le c \in L^{\infty}(\Omega)$ and $f\in L^n(\Omega)$, then
\begin{equation}\label{EqEstA.B.P.}
\sup_{\Omega} u(x) \leq \sup _{\partial \Omega} u^{+}(x)+ \mathrm{C}\cdot \operatorname{diam}(\Omega) \cdot\left\|f^{+}\right\|_{L^{n}(\Gamma^+(u))},	
\end{equation}
where $\Omega \subset \mathbb{R}^n$ is a bounded domain, $\Gamma^+(u)\subset \Omega$ is the upper contact set of u, i.e. the set of points where $u$ is non-negative and concave, and $\mathrm{C}>0$ depends only on dimension,  ellipticity constants of $\mathfrak{A}(x)$ and of $L^{\infty}(\Omega)$ norms of the coefficients. Similar estimates can be obtained if $u\in C(\overline{\Omega})\cap W^{2,n}_{loc}(\Omega)$ is a viscosity super-solution of \ref{eqA.B.P.}. Precisely,
	\begin{eqnarray*}
		\inf_{\Omega} u(x) \geq \inf _{\partial \Omega} u^{-}(x) - \mathrm{C}\cdot \text{diam}(\Omega) \cdot\left\|f^{-}\right\|_{L^{n}(\Gamma^+(-u))}.	
	\end{eqnarray*}

The fact that Alexandrov-Bakel'man-Pucci's estimates (in short, A.B.P. estimates) for linear operators depend mainly on the ellipticity constants of the coefficient of matrix and the geometry of $\Omega \subset \mathbb{R}^n$, allowed the extension of such results to the class of fully nonlinear uniformly elliptic equations as follows  $F(D^2 u) = f \in L^{n}(B_1)$ (see \cite{CC-Book} and \cite{Caf} for details). Later, in \cite[Corollary 1.5]{Cabre}, Cabr\'{e} extended this result by exchanging the constant $\text{diam}(\Omega)$ for a more precise geometric quantity, namely, under certain assumptions, the geometric constant $\text{diam}(\Omega)$ in \eqref{EqEstA.B.P.} can be replaced by $|\Omega|^{\frac{1}{n}}$, where $|\cdot|$ denotes the $n$-dimensional Lebesgue measure. Precisely,
$$
 \mathcal{L}u(x) = f(x) \quad \text{in} \quad \Omega \quad \Rightarrow \quad \sup_{\Omega} u \leq \limsup_{x \to \partial \Omega} u^+(x) + \mathrm{C}\left|\Omega\right|^{\frac{1}{n}}\left|\left|f\right|\right|_{L^n(\Omega)}.
$$
As a consequence, the result proved by Cabr\'{e} makes A.B.P. estimates valid also for domains that are not necessarily bounded.

\medskip

In the context of PDEs in divergence form (but solutions are understood in the viscosity sense), we can cite the fundamental work of Argiolas, Charro, and Peral in \cite[Theorem 2.5]{Argiolas}, in which the authors prove A.B.P. estimates for a class of operators that includes the $p-$Laplacian operator. To be more precise,
\begin{equation}\label{Eq-p-Lapla-type}
-\text{div}(\mathbf{F}(D  u)) = f(x) \quad \text{in} \quad \Omega \quad \Rightarrow \quad 	\sup_{\Omega} u(x) \leq \sup _{\partial \Omega} u^{+}(x)+ \mathrm{C}\cdot \text{diam}(\Omega) \cdot\left\|f^{+}\right\|^{\frac{1}{\alpha}}_{L^{n}(\Gamma^+(u))},	
\end{equation}
where $\alpha>0$ is a structural constant, i.e., it depends only on the properties of the field $\mathbf{F}$ (a $C^1$ monotone vector field under some suitable conditions). Moreover, in the particular case where the operator is the $p$-Laplacian, we have that $\alpha = p-1$.

\medskip

Recently, in \cite[Theorem 2.4.]{BJrDaST2025} Bezerra J\'{u}nior \textit{et al} addressed a version of the $L^{\infty}$-estimate for weak/viscosity solutions of $p-$Laplacian type equations of the form
	\begin{equation}\label{Eq-Quasil-Oper}
		-\operatorname{div} \mathfrak{a}\left(x, D  u\right)=f(x) \quad  \text{in} \quad \Omega,
	\end{equation}
	where $1<p<\infty$, $\Omega \subset \mathbb{R}^{n}$ ($n \geqslant 2$) is a suitable domain (possibly unbounded with finite measure), $f \in L^{q}(\Omega)$ with $q> \frac{n}{p}$ and $q \ge  \frac{p}{p-1}$ and $\mathfrak{a}:\Omega\times \mathbb{R}^n\rightarrow \mathbb{R}^n$ is a continuous vector field satisfying certain structural properties. Precisely, the authors obtain
$$ \sup_{\Omega} u(x) \leq \sup _{\partial \Omega} u^{+}(x)+\mathrm{C}\cdot |\Omega|^{\frac{pq-n}{nq(p-1)}} \cdot\left\|f^{+}\right\|_{L^{q}\left(\left\{\psi=u_{\psi}\right\} \cap \Omega\right)}^{\frac{1}{p-1}}	
$$
where $\displaystyle \left(\psi=u-\sup_{\partial \Omega} u^{+}\right)^{+}$. Moreover, in contrast to the estimates proved for operators in non-divergence form (see \textit{e.g.} \eqref{EqEstA.B.P.}), $u_{\psi}$ is not the upper contact set of $\psi$, but the solution of an obstacle problem related to \eqref{Eq-Quasil-Oper}. In effect, the authors' strategy consists of proving that $u_{\psi}$ is itself a sub-solution of the equation with $f$ located in the contact set $\{\psi = u_{\psi}\} \cap \Omega$.

\medskip

In the manuscript \cite{Charro13}, the authors establish $L^\infty$ bounds and derive estimates for the modulus of continuity of solutions to the Poisson problem involving the normalized infinity and $p$-Laplacian, namely,
\[
- \Delta^{\mathrm{N}}_p u(x) = f(x) \quad \text{in} \quad \Omega, \quad \text{for } \quad  1 < p \leq \infty.
\]
Within this framework, the authors present a robust family of results that depend continuously on the parameter $p$. Furthermore, they demonstrate the failure of the classical Alexandrov-Bakelman-Pucci estimate for the normalized infinity Laplacian and propose alternative estimates.

\begin{theorem}[{\bf  \cite[Theorem 3]{Charro13}}]
Let $1 < p < \infty$, $f \in C^0(\Omega)$, and consider $u \in C^0(\Omega)$ that satisfies
\[
- \Delta^{\mathrm{N}}_p u(x) \leq f(x) \quad \text{in } \quad \Omega
\]
in the viscosity sense. Then, we have
\[
\displaystyle \left( \sup_{\Omega} u - \sup_{\partial \Omega} u^+ \right)^2
\leq 
\frac{2p\, d^2}{p - 1}
\int\limits_{ \sup\limits_{\partial \Omega} u^+}^{\sup\limits_{\Omega} u}
f^+ \cdot \mathbf{1}_{C^+(u^+)} \, \|f\|_{L^\infty(\{u^+ = r\})} \, dr,
\]
where $d = \operatorname{diam}(\Omega)$, $u^+$ is extended by $0$ to the whole $\mathbb{R}^n$, and $\mathrm{C}^+(u^+)$ is the contact set with $\mathcal{A} = \Omega^* = \operatorname{conv}(\Omega)$, the convex hull of $\Omega$.

Analogously, whenever
\[
- \Delta^{\mathrm{N}}_p u(x)\geq f(x) \quad \text{in } \Omega
\]
in the viscosity sense, the following estimate holds:
\[
\left( \sup_{\Omega} u^- - \sup_{\partial \Omega} u^- \right)^2
\leq 
\frac{2p\, d^2}{p - 1}
\int\limits_{\sup\limits_{\partial \Omega} u^-}^{\sup\limits_{\Omega} u^-}
f^- \cdot \mathbf{1}_{C^+(u^-)} \, \|f\|_{L^\infty(\{u^- = r\})} \, dr.
\]
\end{theorem}

\medskip

In concluding our state-of-the-art on $L^{\infty}-$estimates, Baasandorj et al.~\cite{BBLL24} recently 
established an A.B.P.-type estimate for a broad class of singular or degenerate 
fully nonlinear operators of the form 
\begin{equation}\label{Phi-FullyNon}
    \Phi(x,|D u(x)|)F(D^2u(x)) = f(x) \quad \text{in } \quad \Omega \subset \mathbb{R}^n,
\end{equation}
where, for $\mathrm{X}, \mathrm{Y} \in \text{Sym}(n)$ with $\mathrm{Y} \geq 0$ (in the matrix sense),
\begin{equation}
   \mathcal{P}^{-}_{\lambda,\Lambda}(\mathrm{Y}) \,\leq\, F(\mathrm{X}+\mathrm{Y})-F(\mathrm{Y}) \,\leq\, \mathcal{P}^{+}_{\lambda,\Lambda}(\mathrm{Y}), 
\end{equation}
where
$$
\mathcal{P}_{\lambda, \Lambda}^+(\mathrm{M}) \coloneqq \sup_{\mathbf{A} \in \mathfrak{A}_{\lambda, \Lambda}} \operatorname{tr}(\mathbf{A} \mathrm{M}) \quad \text{and} \quad 
\mathcal{P}_{\lambda, \Lambda}^-(\mathrm{M}) \coloneqq \inf_{\mathbf{A} \in \mathfrak{A}_{\lambda, \Lambda}} \operatorname{tr}(\mathbf{A} \mathrm{M})
$$
denote the \textit{Pucci extremal operators}, for $0< \lambda \leq \Lambda$ (ellipticity constants)
$$
\mathfrak{A}_{\lambda, \Lambda} \coloneqq \{\mathbf{A} \in \text{Sym}(n): \lambda \mathrm{Id}_n \leq \mathbf{A} \leq \Lambda \mathrm{Id}_n\}.
$$
and $\Phi : \Omega \times (0,\infty) \to [0,\infty)$ is a continuous function satisfying:
\begin{itemize}
    \item[\checkmark] There exist constants $s(\Phi) \geq i(\Phi) > -1$ such that the mapping 
    $t \mapsto \frac{\Phi(x,t)}{t^{\,i(\Phi)}}$ is almost non-decreasing with constant ${L} \geq 1$ 
    on $(0,\infty)$, while $t \mapsto \frac{\Phi(x,t)}{t^{\,s(\Phi)}}$ is almost non-increasing 
    with the same constant $L$.
    \item[\checkmark] There exist constants $0 < \nu_0 \leq \nu_1$ such that 
    $\nu_0 \leq \Phi(x,1) \leq \nu_1$ for all $x \in \Omega$.
\end{itemize}

Under these assumptions, the authors proved that if $u \in C^0(\Omega)$ is a viscosity 
subsolution of \eqref{Phi-FullyNon} in $\{u > 0\}$, then
\begin{equation}
    \sup_{\Omega} u \;\leq\; \sup_{\partial\Omega} u^{+} 
    + \mathrm{C}_0 \, \mathrm{diam}(\Omega) 
    \max \Bigg\{ 
        \bigl\| f^{-} \bigr\|^{\tfrac{1}{i(\Phi)+1}}_{L^n(\Gamma^{+}(u^{+}))}, \,
        \bigl\| f^{-} \bigr\|^{\tfrac{1}{s(\Phi)+1}}_{L^n(\Gamma^{+}(u^{+}))} 
    \Bigg\} + 1, 
\end{equation}
for some constant $\mathrm{C}_0>0$ depending only on $n,\lambda,i(\Phi),s(\Phi),L$, and $\nu_0$.

Finally, unlike the results in the non-variational setting (see \eqref{eqA.B.P.}, \eqref{Eq-p-Lapla-type}, and \eqref{Phi-FullyNon}), the framework developed in Theorems~\ref{Theorem3.1} and~\ref{A.B.P.caseforp} provides key advantages, introducing a novel approach to the analysis of A.B.P. estimates in the quasilinear context \eqref{1.1}. Moreover, the method is sufficiently flexible to yield boundedness estimates even for solutions of other fully nonlinear degenerate or singular models related to the normalized \(p\)-Laplacian.

\bigskip

The Strong Maximum Principle (SMP for short) for the normalized $p$-Laplacian has been established in several settings. In Euclidean domains, K\"{u}hn (\cite[Theorem 3.2]{kuehn2022}) proved that viscosity subsolutions of $-\Delta_p^{\mathrm{N}} u(x) \le 0$ that attain an interior maximum are constant, and showed that the SMP is equivalent to the strong comparison principle (Theorem~3.7). Moreover, a Hopf-type lemma under an interior sphere condition was also obtained in \cite[Lemma 3.14]{kuehn2022}. More generally, Goffi and Pediconi in \cite[Theorem 4.2]{goffi-pediconi2021}, derived the SMP for degenerate fully nonlinear equations on Riemannian manifolds, which applies to $x \mapsto \Delta_p^{\mathrm{N}} u(x)$ in $\mathbb{R}^n$ as a special case.

Regarding Liouville-type theorems, Alcantara, da Silva, and S\'a (\cite[Theorems~1.13 and 1.14]{alcantara2025}) considered weighted equations of the form
\[
|\nabla u|^\gamma \, \Delta_p^{\mathrm{N}} u(x) = \mathfrak{a}(x) u^m(x) \quad \text{in } \mathbb{R}^n,
\]
proving that nonnegative viscosity subsolutions are constant when $0<m < \gamma+1$ under suitable sign conditions on $\mathfrak{a}$ provided that
\[
\limsup_{|x| \to \infty} \frac{u(x)}{|x|^{\beta}} <\frac{1}{\beta^{\frac{\gamma+1}{\gamma+1-m}}}\left[\frac{\lambda_0}{n-1+(p-1)(\beta-1)}\right]^{\frac{1}{\gamma+1-m}},
\]
where $\beta = \beta(\gamma, m) := \frac{\gamma + 2}{\gamma + 1 - m}$ and $\displaystyle 0< \lambda_0 = \inf_{\mathbb{R}^n} \mathfrak{a}$. In a more general context, Bardi and Cesaroni (\cite[Theorems~2.1--2.2]{bardi-cesaroni2016}) established Liouville properties for a wide class of degenerate fully nonlinear elliptic operators in $\mathbb{R}^n$, covering the normalized $p$-Laplacian when zero- and first-order terms are absent. In \cite{BisVo23}, the authors established Liouville-type theorems for equations involving the $(3-\gamma)$-homogeneous infinity-Laplacian with gradient terms of the form
\[
\Delta^\gamma_\infty u + q(x)\cdot \nabla u \, |\nabla u|^{2-\gamma} + f(x,u) = 0 \quad \text{in } \mathbb{R}^d, \quad \text{for} \quad \Delta^\gamma_\infty u \defeq  \frac{1}{|Du|^{\gamma}} \Delta_{\infty} u\,\, \,\text{with}\,\,\,\gamma \in [0, 2].
\]
They proved, for instance, that any bounded-below Lipschitz supersolution to 
\[
-\Delta^\gamma_\infty u -\mathrm{c}.|\nabla u|^{4-\gamma} = 0 \quad \text{in} \quad \mathbb{R}^d, \,\,\,(\text{for}\,\,\, \mathrm{c}\leq 0)
\]
is necessarily constant, and under suitable decay of $q(x)$ given by $\displaystyle \limsup_{|x|\to+\infty}(q(x)\cdot x)_+<1$, the same result holds for
\[
\Delta^\gamma_\infty u + q(x)\cdot \nabla u \, |\nabla u|^{2-\gamma} = 0 \quad \text{in} \quad \mathbb{R}^d.
\]
 Motivated by their work, we extend such Liouville-type theorems to our operator in more general settings.


\section{Preliminaries}\label{Section2}

This section gathers some definitions and auxiliary results
used throughout this manuscript. We denote by $\mathrm{B}_r(x_0)$ the ball of radius $r$ centered at $x_0$. In what follows, we define viscosity solutions. 

\begin{definition}[{\bf Viscosity Solutions}]\label{DefViscSol}
Let $\Omega\subset \mathbb R^n$ be an open set. An upper (respectively, a lower) semicontinuous function \( u: \Omega\to \mathbb{R}\) is called a viscosity subsolution (respectively, supersolution) of
\begin{equation}\label{Vis1}
-\Psi(x,|D  u|)\Delta_{p}^{\mathrm{N}} u+\mathscr{H}(x,D  u)= f(x) \quad \text{in} \quad \Omega,
\end{equation}
if, for every \(x_{0}\in\Omega\) one of the following conditions hold: 

\begin{enumerate}
\item Either for all \( \phi \in \mathrm{C}^{2}(\Omega) \) and \( u - \phi \) attains a local maximum (respectively, minimum) at \( x_{0} \in \Omega \) and \( D  \phi(x_{0}) \neq 0 \), then
\[
-\Psi(x,|D  \phi(x_0)|)\Delta_{p}^{\mathrm{N}} \phi(x_{0})+\mathscr{H}(x_{0},D \phi(x_{0}))\leq f(x_{0})
\]
 \[
\text{(respectively, }\, -\Psi(x,|D  \phi(x_0)|)\Delta_{p}^{\mathrm{N}} \phi(x_{0})+\mathscr{H}(x_{0},D \phi(x_{0}))\geq f(x_{0})\text{)}.
    \]

    \item Or there exists an open ball \(\mathrm{B}_{\delta}(x_{0})\subset \Omega\) for some \(\delta>0\) such that \(u_{|_{\mathrm{B}_{\delta}(x_{0})}}\) is constant  and 
\[
0\leq f(x) \  \text{(respectively,}\, 0\geq f(x) ),\, \text{for all } x\in \mathrm{B}_{\delta}(x_{0}). 
\]
\end{enumerate}
Finally, we say that \( u\in C^{0}(\Omega) \) is a viscosity solution to \eqref{Vis1} if it satisfies the conditions for both subsolution and supersolution.
\end{definition}

\begin{definition}[{\bf Concave envelope}]
Consider a function $u \in C^{0}(\mathscr{A})$, where $\mathscr{A} \subset \mathbb{R}^{n}$ is an open and convex set. The concave envelope of the function $u$ on $\mathscr{A}$ is the function $\Gamma(u): \mathscr{A} \to \mathbb{R}$ defined by
\begin{eqnarray*}
\Gamma(u)(x)&=&\inf\{w(x):w \text{ is a concave function in }\mathscr{A}\text{ such that }w\geq u\}\\
&=&\inf_{\mathscr{L}\in\mathcal{L}}\{\mathscr{L}(x):\mathscr{L}\geq u \text{ in } \mathscr{A}\},
\end{eqnarray*}
where \(\mathcal{L}\) is the set of all affine functions defined in \(\mathbb{R}^{n}\).
\end{definition}
Clearly, $\Gamma(u)$ is a concave function on $\mathscr{A}$. Henceforth, we define the (upper) contact set of $u$ by
\[
\mathrm{C}(u)=\{x\in \mathscr{A}:\Gamma(u)(x)=u(x)\}.
\] 
Further, the points in \(\mathrm{C}(u)\)
are called contact points.\\
\begin{definition}[{\bf Superjets and Subjets}]
Let $\Omega \subseteq \mathbb{R}^n$ be an open set and $u:\Omega \to \mathbb{R}$ be a continuous function. 
The {second-order superjet} of $u$ at $x_0 \in \Omega$ is the set
\begin{align*}
\mathcal J_\Omega^{2,+} u(x_0) :=
\Bigl\{ (\xi,\mathrm{X}) \in \mathbb{R}^n \times \mathrm{Sym}(n) \ &: \ 
u(x) \le u(x_0) + \xi \cdot (x-x_0) + \tfrac{1}{2}(x-x_0)^T \mathrm{X} (x-x_0) 
\\ & \qquad\qquad\qquad+ o(|x-x_0|^2) \;\; \text{as } x \to x_0 \Bigr\}.
\end{align*}

\medskip 

 The {second-order subjet} of $u$ at $x_0 \in \Omega$ is the set
\begin{align*}
\mathcal{J}_\Omega^{2,-} u(x_0) :=
\Bigl\{ (\xi,\mathrm{X}) \in \mathbb{R}^n \times \mathrm{Sym}(n) \ & : \ 
u(x) \ge u(x_0) + \xi \cdot (x-x_0) + \tfrac{1}{2}(x-x_0)^T \mathrm{X} (x-x_0) 
\\&\qquad\qquad\qquad+ o(|x-x_0|^2) \;\; \text{as } x \to x_0 \Bigr\}.
\end{align*}
 In other words, $(\xi,\mathrm{X}) \in \mathcal J_\Omega^{2,+}u(x_0)$ if the quadratic polynomial
\[
\varphi(x) := u(x_0) + \xi \cdot (x-x_0) + \tfrac{1}{2}(x-x_0)^T \mathrm{X} (x-x_0)
\]
{touches $u$ from above at $x_0$}, up to second order and $(\xi,\mathrm{X}) \in \mathcal J_{\Omega}^{2,-}u(x_0)$ if the quadratic polynomial
\[
\phi(x) := u(x_0) + \xi \cdot (x-x_0) + \tfrac{1}{2}(x-x_0)^T \mathrm{X} (x-x_0)
\]
{touches $u$ from below at $x_0$}, up to second order.
\end{definition}

\begin{remark}
On the previous definition:
\begin{itemize}
  \item $(\xi,\mathrm{X}) \in \mathcal J_\Omega^{2,+}u(x_0)$ means that there exists a quadratic function
  $\varphi$ with gradient $p$ and Hessian $X$ at $x_0$ such that $u-\varphi$ attains a local maximum at $x_0$. For a viscosity subsolution, one tests the PDE using jets from $\mathcal J_\Omega^{2,+}u(x_0)$.
  \item $(\xi,\mathrm{X}) \in \mathcal J_\Omega^{2,-}u(x_0)$ means that there exists a quadratic function
  $\Psi$ with gradient $p$ and Hessian $X$ at $x_0$ such that $u-\Psi$ 
  attains a local minimum at $x_0$. For a {viscosity supersolution}, one tests the PDE using jets from $\mathcal  J_\Omega^{2,-}u(x_0)$.
\end{itemize}
\end{remark}

\medskip

We now recall Proposition~1 and Lemma~5 from \cite{Charro13}, which are key tools for establishing our A.B.P.~estimates.

Before stating these results, note that a Lipschitz function 
\(v : \mathcal{M} \to \mathcal{N}\) between two \(C^{1}\) manifolds is differentiable \(\mathcal H^{n-1}\)-a.e. 
We denote by \(\nabla^{\mathcal{M}} v\) its tangential gradient, i.e., 
the linear map from \(\mathrm{T}_x \mathcal{M}\) to \(\mathrm{T}_x \mathcal{N}\) defined by
\[
\nabla^{\mathcal{M}} v \cdot \xi := \lim_{t \to 0} \frac{v(x+t\xi) - v(x)}{t},
\qquad  \xi \in \mathrm{T}_x \mathcal{M}.
\]
Here \(\mathcal H^{n-1}\) denotes the \((n-1)\)-dimensional Hausdorff measure, which coincides with the standard surface measure on \(C^1\) manifolds.

\begin{proposition}\label{Proposition 1}
Let $w$ be a concave $C^{1,1}$ function in a ball $\mathrm{B}_R(0)$ with $w = 0$ on $\partial \mathrm{B}_R(0)$. Then:
\begin{enumerate}
   \item For every level $\displaystyle r \in (0, \sup_{B_R(0)} w)$, we have $\nabla w \neq \vec{0}$ and 
    \[
    \mathrm{M}_r := \{ w = r \}
    \]
    is a $C^{1,1}$ manifold.
    
    \item For a.e. $\displaystyle r \in [0, \sup_{\mathrm{B}_R(0)} w]$, $\mathcal H^{n-1}$-a.e. $x \in \mathrm{M}_r$ is a point of twice differentiability for $w$.
    
    \item Let $\displaystyle r \in (0, \sup_{\mathrm{B}_R(0)} w)$. If $x \in \mathrm{M}_r$ is a point of twice differentiability, then
    \begin{align}\label{eq:second_derivative}
    -D^2 w(x) 
    &= |\nabla w(x)| \sum_{i=1}^{n-1} \kappa_i(x) \, \tau_i(x) \otimes \tau_i(x) \\
    &\quad - \sum_{i=1}^{n-1} \partial^2_{\nu \tau_i} w(x) 
        \big( \tau_i(x) \otimes \nu(x) + \nu(x) \otimes \tau_i(x) \big) \notag \\
    &\quad - \Delta^\mathrm{N}_\infty w(x) \, \nu(x) \otimes \nu(x), \notag
    \end{align}
    where:
    \begin{enumerate}
        \item $\nu(x)$ is the exterior normal to $\mathrm{M}_r$ at $x$, i.e.
        \[
        \nu(x) = - \frac{\nabla w(x)}{|\nabla w(x)|}.
        \]
        
        \item $\{\tau_i(x)\}_{i=1}^{n-1}$ is an orthonormal basis of $\mathrm{T}_x \mathrm{M}_r$ that diagonalizes the Weingarten operator $\nabla^{\mathcal{M}} \nu$ at $x$.
        
        \item $\kappa_i(x)$ are the eigenvalues of $\nabla^{\mathcal{M}} \nu$ at $x$, which are the principal curvatures of $\mathrm{M}_r$ at $x$.
    \end{enumerate}
\end{enumerate}
\end{proposition}

\begin{lemma}\label{lemma5}
Let $\widetilde\Omega$ be a convex domain and $u$ a continuous semiconvex function such that 
\[
u \leq 0 \quad \text{in } \mathbb{R}^n \setminus \widetilde\Omega, \qquad u(x_0) > 0 \text{ for some } x_0 \in \widetilde\Omega.
\]
Let $\Gamma_{\sigma}(u^+)$ be the concave envelope of $u^+$ (extended by $0$ outside $\widetilde\Omega$) in
\[
\widetilde\Omega_{\sigma} = \{x \in \mathbb{R}^n : \operatorname{dist}(x,\widetilde\Omega) \leq \sigma \}.
\]
Then $\Gamma_{\sigma}(u)$ is $C^{1,1}_{\mathrm{loc}}$ in $\widetilde\Omega_{\sigma}$ (and hence twice differentiable almost everywhere). Moreover,
\[
\{x \in \widetilde\Omega : \det D^2 \Gamma_{\sigma}(u)(x) \neq 0 \} \subset \mathrm{C}^+_{\sigma}(u),
\]
where $\mathrm{C}^+_{\sigma}(u)$ is the set of points in $\widetilde\Omega_{\sigma}$ where $u = \Gamma_{\sigma}(u)$.
\end{lemma}

\medskip

Another key result is a Gauss--Bonnet theorem for \(C^{1,1}\) convex sets. Its proof follows directly from the Area Formula for rectifiable sets (see \cite[Corollary~3.2.20]{federer}).

\begin{theorem}[{\bf Gauss-Bonnet}]\label{theorem9}
Let $\mathcal{K} \subset \mathbb{R}^n$ be a $C^{1,1}$ convex set and let $\nu_{\mathcal{K}}$ denote its outer normal. Then
\begin{equation}\label{eq:gauss-bonnet}
\int_{\partial \mathcal{K}} \prod_{i=1}^{n-1} \kappa_i(x)\, d \mathcal H^{n-1}
= \int_{\partial \mathcal{K}} \det\!\big( \nabla^{\partial \mathcal{K}} \nu_{\mathcal{K}} \big)\, d\mathcal{H}^{n-1}
= \mathcal H^{n-1}\!\big(\partial \mathrm{B}_1(0)\big).
\end{equation}
\end{theorem}
\noindent Next,  we recall the Ishii–Lions lemma, which is a fundamental tool for establishing compactness and existence results in the theory of partial differential equations. For further details, we refer the reader to \cite[Theorem 3.2]{CIL}.

\begin{lemma}[\bf Ishii-Lions Lemma]\label{IshiiLions}
Let $u_i$ be a upper semicontinuous function in $\mathrm{B}_1$ for $i=1,\ldots,k$. Let $\varphi$ be defined on $(\mathrm{B}_1)^{k}$ and such that the function 
\[
(x_1,\dots,x_k) \to \varphi(x_1,\dots,x_k) 
\]
is twice continuously differentiable in $(x_1,\dots,x_k) \in (\mathrm{B}_1)^k$. Suppose that
\[
w(x_1,\dots,x_k) := u_1(x_1) + \cdots + u_k(x_k) - \varphi(x_1,\dots,x_k)
\]
attains a local maximum at $(\bar{x}_1,\dots,\bar{x}_k) \in (\mathrm{B}_1)^{k}$. Assume, moreover, that there exists an $r > 0$ such that for every $\mathrm{M}_{\star} > 0$ there is a constant $\mathrm{C}_{\star}$ such that for $i=1,\ldots,k$,
\[
b_i \leq \mathrm{C}_{\star} \quad \text{whenever } (\vec{q}_i,X_i) \in \mathcal{J}^{2,+}_{\mathrm{B}_{1}}u_i(x_i),
\]
\[
|x_i - \bar{x}_i|  \leq r, \quad \text{and} \quad |u_i(x_i)| + |\vec{q}_i| + \|\mathrm{X}_i\| \leq \mathrm{C}_{\star}.
\]
Then for each $\varepsilon > 0$, there exist $\mathrm{X}_i \in \mathrm{Sym}(n)$ such that:
\begin{itemize}
\item[(i)] $(D_{x_i}\varphi(\bar{x}_1,\ldots,\bar{x}_k),\mathrm{X}_i) \in \overline{\mathcal{J}}^{2,+}_{\mathrm{B}_{1}}u_i(\bar{x}_i)$ for $i=1,\dots,k$,
\item[(ii)] $- \left( \frac{1}{\varepsilon} + \|\mathrm{A}\| \right) \mathrm{Id}_n \leq \begin{pmatrix} \mathrm{X}_1 & & 0 \\ & \ddots & \\ 0 & & \mathrm{X}_k \end{pmatrix} \leq \mathrm{A} + \varepsilon \mathrm{A}^2$,
\end{itemize}
where $\mathrm{A} = D^2 \varphi(\bar{x}_1,\ldots,\bar{x}_k)$.
\end{lemma}

\section{A.B.P. estimate for models with general  $p-$diffusion}\label{Section2}

This section is devoted to the presentation of the Alexandroff–Bakelman–Pucci (A.B.P.) maximum principle for the model \eqref{1.1}. We start by proving the result for the case of $p = \infty$ diffusion.

\begin{proof}[{ \bf Proof of Theorem \ref{Theorem3.1}:}]
We focus on subsolutions only, as the proof for supersolutions is analogous. The argument proceeds in two steps: first, we assume \(u\) is semiconvex; second, we treat the general case via approximation by sup-convolutions.

\textbf{ ${\mathbf 1^{\textbf {st}}}$ Step:} Assume that \(u\) is semiconvex. Without loss of generality, we may suppose that 
$
\sup_{\partial \Omega} u^{+}=0$
(otherwise, replace \(u\) by \(u - \sup_{\partial \Omega} u^{+}\)), and that there exists \(x_{0} \in \Omega\) with \(u(x_{0}) > 0\) (otherwise, the desired inequality already holds). Extend \(u^{+}\) outside \(\Omega\) by zero; then \(u^{+}\) remains semiconvex and satisfies, in the viscosity sense,
\begin{eqnarray}
-\Psi(x,|D  u^{+}|)\Delta_{\infty}^{\mathrm{N}}u^{+}+\mathscr{H}(x,D  u^{+})\leq f^{+}(x)\qquad \mbox{in }\mathbb{R}^{n}. \label{eq1withu}
\end{eqnarray}

Given \(r>0\), define \(\Gamma_{r}(u^{+})\) as the concave envelope of \(u^{+}\) in  
\[
\Omega^{*}_{r}=\{x\in\mathbb{R}^{n}: \operatorname{dist}(x,\Omega^{*})\leq r\},
\]
where \(\Omega^{*}=\operatorname{conv}(\Omega)\) denotes the convex hull of \(\Omega\).  
By Lemma \ref{lemma5}, the semiconvexity of \(u^{+}\) ensures that \(\Gamma_{r}(u^{+})\in C^{1,1}(\Omega^{*}_{r})\) for every \(r>0\). Thus, we can consider the integral
\begin{equation}\label{integralI}
\mathscr{I}=\int\limits_{D \Gamma_{r}(u^{+})(\Omega^{*}_{r})}\min\{|z|^{2+i_{\Psi}-\sigma-n},\, |z|^{2+s_{\Psi}-n}\}\,dz.
\end{equation}

For convenience, set 
\[
\mathrm{M}=\sup_{\Omega_{r}^{*}}\Gamma_{r}(u^{+})=\sup_{\Omega}u^{+}=\sup_{\Omega}u,
\]
denote by \(\mathrm{C}_{r}(u^{+})\) the contact set of \(u^{+}\) with \(\Gamma_{r}(u^{+})\), and define  
\[
\mathrm{A}_{r}=\mathrm{C}_{r}(u^{+})\cap\{x\in \Omega^{*}_{r}:\exists\,\, D^{2}\Gamma_{r}(u^{+})(x)\}.
\]

We aim to estimate the integral \(\mathscr{I}\) from above and below. For the upper bound, define 
\[
\mathfrak{h}(t) = \min\{t^{i_{\Psi}-\sigma},\, t^{s_{\Psi}}\}.
\] 
Applying the Change of Variables Formula for Lipschitz functions \cite[Theorem 3.9]{EG15} and integration over level sets \cite[Theorem 3.13]{EG15}, we obtain
\begin{eqnarray*}
\mathscr{I} &\leq& \int_{\Omega^{*}_{r}} |D \Gamma_{r}(u^{+})|^{2-n} \mathfrak{h}(|D \Gamma_{r}(u^{+})|) 
\det(-D^{2}\Gamma_{r}(u^{+})) \chi_{\mathrm{A}_{r}} \, dx \\
&=& \int_{0}^{\mathrm{M}} \int_{\{\Gamma_{r}(u^{+})=\tau\}} 
|D \Gamma_{r}(u^{+})|^{1-n} \mathfrak{h}(|D \Gamma_{r}(u^{+})|) 
\det(-D^{2}\Gamma_{r}(u^{+})) \chi_{\mathrm{A}_{r}} \, d\mathcal{H}^{n-1} d\tau. \label{estimateofI1}
\end{eqnarray*}

By item 3 of Proposition \ref{Proposition 1}, for a.e. \(\tau\in (0,\mathrm{M})\),
\[
\det(-D^{2}\Gamma_{r}(u^{+})(x)) \leq -\Delta_{\infty}^{\mathrm{N}}\Gamma_{r}(u^{+})(x) |D \Gamma_{r}(u^{+})(x)|^{\,n-1} \prod_{i=1}^{n-1} \kappa_i(x),
\]
for \(\mathcal{H}^{n-1}\)-a.e. \(x \in \{\Gamma_{r}(u^{+})=\tau\}\), where \(\kappa_i(x)\) (\(i=1,\ldots,n-1\)) denote the principal curvatures of \(\{\Gamma_{r}(u^{+})=\tau\}\) at \(x\). Using this inequality together with \eqref{eq1withu}, we obtain
\begin{eqnarray}
\mathscr{I} &\leq& \int_{0}^{\mathrm{M}} \int_{\{\Gamma_{r}(u^{+})=\tau\}} 
\mathfrak{h}(|D \Gamma_{r}(u^{+})|) \frac{(f^{+}+\varrho |D \Gamma_{r}(u^{+})|^{\sigma}) \chi_{\mathrm{A}_{r}}}{\Psi(x,|D \Gamma_{r}(u^{+})|)} 
\prod_{i=1}^{n-1} \kappa_i(x) \, d\mathcal{H}^{n-1} d\tau \nonumber\\
&\leq& \frac{\mathfrak{L}_1}{\mathfrak{a}} \int_{0}^{\mathrm{M}} \int_{\{\Gamma_{r}(u^{+})=\tau\}} 
(f^{+} + \varrho) \chi_{\mathrm{A}_{r}} \prod_{i=1}^{n-1} \kappa_i(x) \, d\mathcal{H}^{n-1} d\tau \label{estimateofI2} \\
&\leq& \frac{\mathfrak{L}_1}{\mathfrak{a}} \int_{0}^{\mathrm{M}} \| (f^{+}+\varrho)\chi_{\mathrm{C}_{r}(u^{+})} \|_{L^{\infty}(\{u^{+}=\tau\})} 
\int_{\{\Gamma_{r}(u^{+})=\tau\}} \prod_{i=1}^{n-1} \kappa_i(x) \, d\mathcal{H}^{n-1} d\tau \nonumber\\
&=& \frac{\mathfrak{L}_1 \mathcal{H}^{n-1}(\partial \mathrm{B}_{1})}{\mathfrak{a}} 
\int_{0}^{\mathrm{M}} \| (f^{+}+\varrho)\chi_{\mathrm{C}_{r}(u^{+})} \|_{L^{\infty}(\{u^{+}=\tau\})} d\tau, \label{estimateofI3}
\end{eqnarray}
where \eqref{estimateofI2} uses Remark \eqref{obsdePsi} and \eqref{estimateofI3} employs the Gauss--Bonnet Theorem for \(C^{1,1}\) convex sets (Theorem \ref{theorem9}).


Next, following \cite[Theorem 1]{Charro13} and using \eqref{estimateofI3}, one verifies that
\[
\int_{\partial K_{\bar{x},r}(\Omega_{r}^{*})} |z|^{2-n} \mathfrak{h}(|z|) \, dz
\leq \frac{\mathfrak{L}_1 \mathcal{H}^{n-1}(\partial \mathrm{B}_{1})}{\mathfrak{a}}
\int_0^{\mathrm{M}} \| (f^{+}+\varrho)\chi_{\mathrm{C}_{r}(u^{+})} \|_{L^{\infty}(\{u^{+}=\tau\})} d\tau,
\]
where \(K_{\bar{x},r}\) is the concave cone with vertex \((\bar{x}, u^{+}(\bar{x}))\) and base \(\Omega_{r}^{*}\), for \(\bar{x}\) a maximum point of \(u^{+}\) in \(\Omega^{*}\), and \(\partial K_{\bar{x},r}\) its super-differential. Letting \(r \to 0^+\), we obtain
\begin{equation}\label{estimateofI4}
\int_{\partial K_{\bar{x}}(\Omega^{*})} |z|^{2-n} \mathfrak{h}(|z|) \, dz
\leq \frac{\mathfrak{L}_1 \mathcal{H}^{n-1}(\partial \mathrm{B}_{1})}{\mathfrak{a}}
\int_0^{\mathrm{M}} \| (f^{+}+\varrho)\chi_{\mathrm{C}(u^{+})} \|_{L^{\infty}(\{u^{+}=\tau\})} d\tau,
\end{equation}
where \(K_{\bar{x}}\) is the concave cone with vertex \((\bar{x}, u^{+}(\bar{x}))\) and base \(\Omega^{*}\), and \(\partial K_{\bar{x}}\) its super-differential.  
It follows from \cite[Theorem 1]{Charro13} that 
\(\mathrm{B}_{\frac{\mathrm{M}}{d}} \subset \partial K_{\bar{x}}(\Omega^{*})\), where \(d = \operatorname{diam}(\Omega)\).  
Applying the Polar Coordinates Formula \cite[Theorem 3.12]{EG15}, we obtain
\begin{eqnarray}
\int_{\partial K_{\bar{x}}(\Omega^{*})} |z|^{2-n} \mathfrak{h}(|z|) \, dz 
&\geq& \int_0^{\frac{\mathrm{M}}{d}} \int_{\partial \mathrm{B}_1} \rho^{2-n} \mathfrak{h}(\rho) \rho^{\,n-1} \, d\mathcal{H}^{n-1} d\rho \nonumber\\
&\geq& \frac{\mathcal{H}^{n-1}(\partial \mathrm{B}_1)}{2+s_{\Psi}} \left( \frac{\mathrm{M}}{d} \right)^{2+s_{\Psi}}, \label{estimateofI5}
\end{eqnarray}
whenever \(\frac{\mathrm{M}}{d} \leq 1\). Otherwise, we have
\begin{eqnarray}
\int_{\partial K_{\bar{x}}(\Omega^{*})} |z|^{2-n} \mathfrak{h}(|z|) \, dz 
&\geq& \int_0^{\frac{\mathrm{M}}{d}} \int_{\partial \mathrm{B}_1} \rho^{1 + s_{\Psi}} \, d\mathcal{H}^{n-1} d\rho + 
\int_1^{\frac{\mathrm{M}}{d}} \rho^{1 + i_{\Psi} - \sigma} \, d\rho \nonumber\\
&=& \frac{\mathcal{H}^{n-1}(\partial \mathrm{B}_1)}{2 + s_{\Psi}} + \frac{\mathcal{H}^{n-1}(\partial \mathrm{B}_1)}{2 + i_{\Psi} - \sigma} \left[ \left( \frac{\mathrm{M}}{d} \right)^{2 + i_{\Psi} - \sigma} - 1 \right] \nonumber\\
&=& \frac{\mathcal{H}^{n-1}(\partial \mathrm{B}_1)}{2 + s_{\Psi}} \left( \frac{\mathrm{M}}{d} \right)^{2 + i_{\Psi} - \sigma}, \label{estimateofI7}
\end{eqnarray}
where in \eqref{estimateofI7} we used the fact that \(-(\mathrm{M}/d)^{2+i_{\Psi}-\sigma} < -1\).  
Combining \eqref{estimateofI3}, \eqref{estimateofI5}, and \eqref{estimateofI7} yields the desired estimate.

\medskip

\textbf{ ${\mathbf 2^{\textbf {nd}}}$ Step:}For the general case, we introduce, for each \(\varepsilon>0\), the \textit{sup-convolution} of \(u\), defined by
\[
u^{\varepsilon}(x) = \sup_{y \in \overline{\Omega}} \left\{ u(y) - \frac{1}{2\varepsilon} |x-y|^2 \right\}.
\]
Note that \(u^{\varepsilon}\) is semiconvex, \(u^{\varepsilon} \geq u\) for all \(\varepsilon>0\), and it can be verified (cf. \cite[Lemma A.3]{CCKS}) that
\[
-\Psi_{\varepsilon}(x,|D u^{\varepsilon}|)\Delta_{\infty}^{\mathrm{N}} u^{\varepsilon} + \mathscr{H}_{\varepsilon}(x,D u^{\varepsilon}) \leq f_{\varepsilon}(x) \quad \text{in } \Omega_{\varepsilon},
\]
where \(\displaystyle f_{\varepsilon}(x) = \sup_{\mathrm{B}_{r_{\varepsilon}}(x)} f\), with \(r_{\varepsilon} = 2 (\varepsilon \|u\|_{L^{\infty}(\Omega)})^{1/2}\), and \(\Omega_{\varepsilon} = \{ x \in \Omega : \operatorname{dist}(x,\partial\Omega) > r_{\varepsilon} \}\).  

We further define \(\displaystyle \mathrm{M}_{\varepsilon} = \sup_{\Omega_{\varepsilon}} u^{\varepsilon}\), \(\mathrm{M}_{\varepsilon}^{+} = \displaystyle \sup_{\partial \Omega_{\varepsilon}} (u^{\varepsilon})^{+}\), and \(d_{\varepsilon} = \operatorname{diam}(\Omega_{\varepsilon})\).  
Using \textbf{Step 1}, we then obtain
\begin{equation}\label{estimativaparauepsilon}
\min\Bigg\{ \left( \frac{\mathrm{M}_{\varepsilon} - \mathrm{M}_{\varepsilon}^{+}}{d_{\varepsilon}} \right)^{2+i_{\Psi}-\sigma}, \left( \frac{\mathrm{M}_{\varepsilon} - \mathrm{M}_{\varepsilon}^{+}}{d_{\varepsilon}} \right)^{2+s_{\Psi}} \Bigg\} 
\leq \frac{(2+s_{\Psi})\mathfrak{L}_1}{\mathfrak{a}} 
\int_{\mathrm{M}_{\varepsilon}^{+}}^{\mathrm{M}_{\varepsilon}} \| (f^{+}_{\varepsilon} + \varrho) \chi_{\mathrm{C}((u^{\varepsilon})^{+})} \|_{L^{\infty}(\{(u^{\varepsilon})^{+}=\tau\})} \, d\tau.
\end{equation}

Moreover, it is known that \(u^{\varepsilon} \to u\) uniformly in \(\overline{\Omega}\) as \(\varepsilon \to 0^+\) (cf. \cite[Lemma A.2]{CCKS}), so that
\[
\min\Bigg\{ \left( \frac{\mathrm{M}_{\varepsilon} - \mathrm{M}_{\varepsilon}^{+}}{d_{\varepsilon}} \right)^{2+i_{\Psi}-\sigma}, \left( \frac{\mathrm{M}_{\varepsilon} - \mathrm{M}_{\varepsilon}^{+}}{d_{\varepsilon}} \right)^{2+s_{\Psi}} \Bigg\} \to 
\min\Bigg\{ \left( \frac{\mathrm{M} - \mathrm{M}^{+}}{d} \right)^{2+i_{\Psi}-\sigma}, \left( \frac{\mathrm{M} - \mathrm{M}^{+}}{d} \right)^{2+s_{\Psi}} \Bigg\},
\]
where \(\mathrm{M}^{+} = \sup\limits_{\partial \Omega} u^{+}\).  

Furthermore, as in \cite{Charro13}, one can verify that
\[
\limsup_{\varepsilon \to 0^+} \| (f^{+}_{\varepsilon} + \varrho) \chi_{\mathrm{C}((u^{\varepsilon})^{+})} \|_{L^{\infty}(\{(u^{\varepsilon})^{+} = \tau\})} 
\leq \| (f^{+} + \varrho) \chi_{\mathrm{C}(u^{+})} \|_{L^{\infty}(\{u^{+} = \tau\})}.
\]

Thus, passing to the limit \(\varepsilon \to 0^+\) in \eqref{estimativaparauepsilon} yields the desired estimate. This completes the proof.
\end{proof}

\medskip

Next, we establish the Alexandroff–Bakelman–Pucci (A.B.P.) maximum principle for the model \eqref{1.1}, in the case of $p < \infty$.
\begin{proof}[{\bf Proof of Theorem \ref{A.B.P.caseforp}:}]
The proof follows the same strategy as in Theorem \ref{Theorem3.1}. First, observe that
\begin{eqnarray}\label{p-laplacian}
\Delta_p^{\mathrm{N}} u &=& \frac{1}{p}\operatorname{tr}\left[\left(\mathrm{Id}_n+(p-2)\frac{D u\otimes D u}{|D u|^2}\right)D^2u\right]\nonumber\\[0.2cm]
&=& \frac{1}{p}\Delta u + \frac{p-2}{p}\Delta_\infty^{\mathrm{N}} u\nonumber\\[0.2cm]
&=& \frac{1}{p}|D u|\,\operatorname{div}\left(\frac{D u}{|D u|}\right)+\frac{p-1}{p}\Delta_\infty^{\mathrm{N}} u.
\end{eqnarray}
Using the same notation as in Theorem \ref{Theorem3.1}, for almost every contact point \(x \in \mathrm{C}_r(u^+)\) we have
\begin{eqnarray*}
& & -\Psi(x,|D \Gamma_r(u^+)|)\frac{p-1}{p}\Delta^{\mathrm{N}}_\infty \Gamma_r(u^+)(x) + \mathscr{H}(x, D \Gamma_r(u^+))\\
&\leq & -\Psi(x,|D \Gamma_r(u^+)|)\frac{p-1}{p}\Delta^{\mathrm{N}}_\infty \Gamma_r(u^+)(x) 
+ \Psi(x,|D \Gamma_r(u^+)|)\frac{1}{p}|\Gamma_r(u^+)(x)|\sum_{i=1}^{n-1}\kappa_i(x) \\
& & + \mathscr{H}(x,D \Gamma_r(u^+)(x))\\
&=& -\Psi(x,|D \Gamma_r(u^+)|)\Delta_p^{\mathrm{N}} \Gamma_r(u^+)(x) 
+ \mathscr{H}(x,D \Gamma_r(u^+)(x))\\
&\leq & f^+(x),
\end{eqnarray*}
where we used
\[
\operatorname{div}\left(\frac{D \Gamma_{r}(u^{+})(x)}{|D \Gamma_{r}(u^{+})(x)|}\right)=-\sum_{i=1}^{n-1}\kappa_{i}(x).
\]

Hence,
\begin{equation}\label{Aux_equation}
-\Psi(x,|D \Gamma_r(u^+)|)\Delta^{\mathrm{N}}_\infty \Gamma_r(u^+)(x)+\overline{\mathscr{H}}(x,\Gamma_r(u^+))\leq \overline{f}^+(x),
\end{equation}
with
\[
\overline{\mathscr{H}}(x,\xi)=\frac{p}{p-1}\mathscr{H}(x,\xi), 
\qquad 
\overline{f}^+(x)= \frac{p}{p-1}f^+(x).
\]
Moreover,
\[
|\overline{\mathscr{H}}(x,\xi)|=\frac{p}{p-1}|{\mathscr{H}}(x,\xi)|
\leq \frac{p}{p-1}\varrho(x)|\xi|^\sigma.
\]

From here, the arguments of Theorem \ref{Theorem3.1} apply verbatim to derive both upper and lower bounds.  
We only sketch the upper bound, since the lower bound does not depend on the right-hand side of \eqref{Aux_equation}.

Set
\[
\mathfrak{I}= \int_{D \Gamma_{r}(u^{+})(\Omega^{*}_{r})} |z|^{2-n}\mathfrak{h}(|z|)\,dz,
\qquad 
\mathfrak{h}(t)=\min\{t^{i_{\Psi}-\sigma},t^{s_{\Psi}}\}.
\]
Then
\[
\mathfrak{I} \leq  \frac{p}{p-1}\frac{\mathfrak{L}_1\mathcal{H}^{n-1}(\partial \mathrm{B}_{1})}{\mathfrak{a}}
\int_{0}^{\mathrm{M}}\|(f^{+}+\varrho)\chi_{\mathrm{C}_{r}(u^{+})}\|_{L^{\infty}(\{u^{+}=\tau\})}\,d\tau,
\]
and consequently
\[
\int_{\partial K_{\bar{x}}(\Omega^{*})}|z|^{2-n}\mathfrak{h}(z)\,dz
\leq 
\frac{p\,\mathfrak{L}_1\mathcal{H}^{n-1}(\partial \mathrm{B}_{1})}{\mathfrak{a}(p-1)}
\int_{0}^{\mathrm{M}}\|f^{+}\chi_{\mathrm{C}(u^{+})}\|_{L^{\infty}(\{u^{+}=\tau\})}\,d\tau.
\]
The estimate \eqref{estimateofI7} remains valid.

Finally, for the general case,
\[
\min\left\{
\left(\frac{\mathrm{M}-\mathrm{M}^{+}}{d}\right)^{2+i_{\Psi}-\sigma},
\left(\frac{\mathrm{M}-\mathrm{M}^{+}}{d}\right)^{2+s_{\Psi}}
\right\}
\leq 
\frac{p(2+s_{\Psi})\mathfrak{L}_1}{\mathfrak{a}(p-1)}
\|(f^{+}+\varrho)\chi_{\mathrm{C}(u^{+})}\|_{L^{\infty}(\{u^{+}=\tau\})}.
\]
This completes the proof.
\end{proof}

\begin{remark}
Regarding Theorem \ref{A.B.P.caseforp}:
\begin{itemize}
\item The estimates remain stable in \(p\), since \(\frac{p}{p-1}\to 1\) as \(p\to \infty\).
\item Setting \(\Phi(x, t) = t^{p-2} \Psi(x, t)\) yields estimates for the classical \(p\)-Laplacian case. Specifically, if \(u \in C^{0}(\overline{\Omega})\) is a subsolution of
\[
-\Psi(x,|D u|)\Delta_{p} u+\mathscr{H}(x,D u)\leq f(x) \quad \text{in} \quad \Omega,
\]
with \(f\in C^{0}(\Omega)\) and \(1< p<\infty\), then
\begin{equation}\label{Eq3.12}
\sup_{\Omega}u\leq\sup_{\partial\Omega}u^{+}+\mathrm{C}\max\left\{\|f^{+}+\varrho\|_{L^{\infty}(\mathrm{C}(u^{+}))}^{\frac{1}{p-1+i_{\Psi}-\sigma}},\|f^{+}+\varrho\|_{L^{\infty}(\mathrm{C}(u^{+}))}^{\frac{1}{p-1+s_{\Psi}}}\right\},
\end{equation}
where
{\small
\[
\mathrm{C}=\max\left\{\left(\operatorname{diam}(\Omega)^{p+i_{\Psi}-\sigma}\frac{(p+s_{\Psi})\mathfrak{L}_1p}{\mathfrak{a}(p-1)}\right)^{\frac{1}{p-1+i_{\Psi}-\sigma}},\left(\operatorname{diam}(\Omega)^{p+s_{\Psi}}\frac{(p+s_{\Psi})\mathfrak{L}_1p}{\mathfrak{a}(p-1)}\right)^{\frac{1}{p-1+s_{\Psi}}}\right\}.
\]}

In particular, these estimates remain effective even for equations driven by the \(p\)-Laplacian with an additional Hamiltonian term
\[
-\Delta_{p} u+\mathscr{H}(x,D u)\leq f(x) \quad \text{in} \quad \Omega,
\]
where \eqref{Eq3.12} follows under the assumptions \(i_{\Psi} = s_{\Psi} = 0\) and \(\mathfrak{a}= \mathfrak{L} = 1\) (cf.  \cite[Theorem 2.5]{Argiolas} and \cite[Theorem 2.4]{BJrDaST2025}).
\end{itemize}
\end{remark}

\medskip
Next, we recall the definition of $(3-h)$-homogeneous \textit{$\infty-$Laplacian}, given by 
\[
\Delta_{\infty}^{h}u\defeq|D  u|^{-h}\Delta_{\infty}u,
\]
for \(h\in[0,2]\). Note that the infinity-Laplacian is \(\Delta_{\infty}^{h} u\) for \(h=0\) and the normalized infinity-Laplacian when \(h=2\). 

As a direct consequence of the A.B.P. estimate, i.e., Theorem \ref{Theorem3.1}, we obtain the result for the $(3-h)$-homogeneous \textit{$\infty-$Laplacian}.

\begin{corollary}
Consider \(u\in C^{0}(\overline{\Omega})\) be a viscosity solution to
\[
-\Psi(x,|D  u|)\Delta_{\infty}^{h} u+\mathscr{H}(x,D  u)\leq f(x) \quad \text{in} \quad \Omega,
\]
where \(f\in C^{0}(\Omega)\). Then,
\begin{eqnarray*}
\sup_{\Omega}u\leq\sup_{\partial\Omega}u^{+}+\mathrm{C}\max\left\{\|f^{+}+\varrho\|_{L^{\infty}(\mathrm{C}(u^{+}))}^{\frac{1}{3-h+i_{\Psi}-\sigma}},\|f^{+}+\varrho\|_{L^{\infty}(\mathrm{C}(u^{+}))}^{\frac{1}{3-h+s_{\Psi}}}\right\}.
\end{eqnarray*}
Similarity, if \(u\) solves,
\[
-\Psi(x,|D  u|)\Delta_{\infty} u+\mathscr{H}(x,D  u)\geq f(x) \quad \text{in} \quad \Omega,
\]
with \(f\in C^{0}(\Omega)\), then
\[
\sup_{\Omega}u^{-}\leq\sup_{\partial\Omega}u^{-}+\mathrm{C}\max\left\{\|f^{-}+\varrho\|_{L^{\infty}(\mathrm{C}(u^{-}))}^{\frac{1}{3-h+i_{\Psi}-\sigma}},\|f^{-}+\varrho\|_{L^{\infty}(\mathrm{C}(u^{-}))}^{\frac{1}{3-h+s_{\Psi}}}\right\},
\]
where \(\mathrm{C}\) in both cases above is given by
{\small\[
\mathrm{C}=\max\left\{\left(\operatorname{diam}(\Omega)^{4-h+i_{\Psi}-\sigma}\frac{(4-h+s_{\Psi})\mathfrak{L}_1}{\mathfrak{a}}\right)^{\frac{1}{3-h+i_{\Psi}-\sigma}},\left(\operatorname{diam}(\Omega)^{4-h+s_{\Psi}}\frac{(4-h+s_{\Psi})\mathfrak{L}_1}{\mathfrak{a}}\right)^{\frac{1}{3-h+s_{\Psi}}}\right\}.
\]}
\end{corollary}
\begin{proof}
Note that $u$ satisfies, in the viscosity sense,
\[
-\Phi(x,|D  u|)\Delta_{\infty}^{\mathrm{N}} u+\mathscr{H}(x,D  u)\leq f(x) \quad \text{in} \quad \Omega,
\]
where \(\overline{\Psi}(x,t)=t^{2-h}\Psi(x,t)\). In this case, we have that \(\overline{\Psi}\) satisfies the conditions \((\Psi1)-(\Psi2)\) with \(i_{\overline{\Psi}}=(2-h)+i_{\Psi}\), \(s_{\overline{\Psi}}=(2-h)+s_{\Psi}\) and the same constants \(0<\mathfrak{a}\leq \mathfrak{b}\), \(\mathfrak{L}_1\) and \(\mathfrak{L}_{2}\). Thus, by the previous A.B.P. estimate (Theorem \ref{Theorem3.1}), we obtain the desired result.

\end{proof}

\section{The Comparison Principle and existence of solutions}

In this section, we establish a version of the Comparison Principle.

\begin{proof}[{\bf Proof of Theorem \ref{CPG}}]
We will proceed with the proof using a \textit{reductio ad absurdum} argument. Specifically, assume that there exists a constant $\mathrm{M}_0 > 0$ such that
\begin{equation}
\mathrm{M}_0 \coloneqq \sup_{\overline{\Omega}} (\mathfrak{u} - \mathfrak{v}) > 0.\label{condideM0}
\end{equation}
For each $\varepsilon > 0$, define
\begin{equation}\label{maximum}
\mathrm{M}_\varepsilon = \sup_{(x,y) \in \overline{\Omega} \times \overline{\Omega}} \left( \mathfrak{u}(x) - \mathfrak{v}(y) - \frac{1}{2\varepsilon} |x - y|^2 \right) < \infty.
\end{equation}
Consider $(x_\varepsilon, y_\varepsilon) \in \overline{\Omega} \times \overline{\Omega}$ as the point where the supremum \(\mathrm{M}_{\varepsilon}\) is attained. Following the same approach as in \cite[Lemma 3.1]{CIL}, we know that
\begin{equation}\label{CP1}
\lim_{\varepsilon \to 0} \frac{1}{\varepsilon} |x_\varepsilon - y_\varepsilon|^2 = 0 \quad \text{and} \quad \lim_{\varepsilon \to 0} \mathrm{M}_\varepsilon = \mathrm{M}_0.
\end{equation}
In particular,
\begin{equation}\label{CP2}
x_0 \coloneqq \lim_{\varepsilon \to 0} x_\varepsilon = \lim_{\varepsilon \to 0} y_\varepsilon,
\end{equation}
where $\mathfrak{u}(x_0) - \mathfrak{v}(x_0) = \mathrm{M}_0$. From \eqref{condideM0}, it follows that
\[
\sup_{\partial \Omega} (\mathfrak{u} - \mathfrak{v}) \leq 0 < \mathrm{M}_0.
\]

Furthermore, since the maximizer cannot approach the boundary, there exists a compact subset $\Omega_1 \Subset \Omega$ such that $x_\varepsilon, y_\varepsilon \in \Omega_1$ for all sufficiently small $\varepsilon > 0$. By the Ishii-Lions Lemma \ref{IshiiLions}, there exist matrices $\mathrm{X}, \mathrm{Y} \in \text{Sym}(n)$ such that
\begin{equation}\label{CP3}
\left( \frac{x_\varepsilon - y_\varepsilon}{\varepsilon}, \mathrm{X} \right) \in \overline{\mathcal{J}}^{2,+}_{\Omega_1} \mathfrak{u}(x_\varepsilon) \quad \text{and} \quad \left( \frac{y_\varepsilon - x_\varepsilon}{\varepsilon}, \mathrm{Y} \right) \in \overline{\mathcal{J}}^{2,-}_{\Omega_1} \mathfrak{v}(y_\varepsilon),
\end{equation}
and
\begin{equation}\label{CP4}
-\frac{3}{\varepsilon} \begin{pmatrix}
\mathrm{Id}_n & 0 \\
0 & \mathrm{Id}_n
\end{pmatrix} \leq \begin{pmatrix}
\mathrm{X} & 0 \\
0 & -\mathrm{Y}
\end{pmatrix} \leq \frac{3}{\varepsilon} \begin{pmatrix}
\mathrm{Id}_n & -\mathrm{Id}_n \\
-\mathrm{Id}_n & \mathrm{Id}_n
\end{pmatrix}.
\end{equation}
In particular, we obtain $\mathrm{X} \leq \mathrm{Y}$.

Now, given that $\mathfrak{u}$ and $\mathfrak{v}$ are Lipschitz continuous in $\Omega_1$, there exists a positive constant $\mathfrak{L}_0 := \max\{[\mathfrak{u}]_{\mathrm{\mathrm{C}^{0, 1}(\Omega_1)}}, [\mathfrak{v}]_{\mathrm{\mathrm{C}^{0, 1}(\Omega_1)}}\}$ such that
\[
|u(z_1) - u(z_2)| + |v(z_1) - v(z_2)| \leq \mathfrak{L}_0|z_1 - z_2| \quad \text{for all} \quad z_1, z_2 \in \Omega_1.
\]

Moreover, using the inequality
\[
\mathfrak{u}(x_\varepsilon) - \mathfrak{v}(x_\varepsilon) \leq \mathfrak{u}(x_\varepsilon) - \mathfrak{v}(y_\varepsilon) - \frac{1}{2\varepsilon}|x_\varepsilon - y_\varepsilon|^2,
\]
we deduce the following
\begin{equation}\label{Eq3.9}
\frac{1}{\varepsilon}|x_\varepsilon - y_\varepsilon| \leq 2\mathfrak{L}_0.
\end{equation}

Now, setting $\eta_\varepsilon := \frac{x_\varepsilon - y_\varepsilon}{\varepsilon}$ and using \eqref{H}, \eqref{CP-General} and \eqref{CP3}, we derive
\begin{align*}
\mathfrak{f}_1(x_\varepsilon) &\leq \Psi(x_{\varepsilon}, |\eta_\varepsilon|)F_0(\eta_\varepsilon, \mathrm{X}) + \mathfrak{c}(x_\varepsilon) \mathscr{F}(\mathfrak{u}(x_\varepsilon)) + \mathscr{H}(x_\varepsilon,\eta_\varepsilon) \\
&\leq \Psi(y_{\varepsilon}, |\eta_\varepsilon|)F_0(\eta_\varepsilon, \mathrm{Y}) + \mathfrak{c}(x_\varepsilon) \mathscr{F}(\mathfrak{u}(x_\varepsilon)) + \mathscr{H}(x_\varepsilon,\eta_\varepsilon) \\
&\leq \mathfrak{f}_2(y_\varepsilon) - \mathfrak{c}(y_\varepsilon) \mathscr{F}(\mathfrak{v}(y_\varepsilon)) - \mathscr{H}(y_\varepsilon,\eta_\varepsilon,) + \mathfrak{c}(x_\varepsilon) \mathscr{F}(\mathfrak{u}(x_\varepsilon)) + \mathscr{H}(x_\varepsilon,\eta_\varepsilon) \\
&\leq \mathfrak{f}_2(y_\varepsilon) + \mathscr{F}(\mathfrak{v}(y_\varepsilon)) (\mathfrak{c}(x_\varepsilon) - \mathfrak{c}(y_\varepsilon)) + \left[ \min_{\overline{\Omega}} \mathfrak{c} \right] (\mathscr{F}(\mathfrak{u}(x_\varepsilon)) - \mathscr{F}(\mathfrak{v}(y_\varepsilon))) \\
&\quad + \omega(|x_{\varepsilon}-y_{\varepsilon}|)(1+\mathscr{H}_0(|\eta_{\varepsilon}|)).
\end{align*}

Finally, by letting \(\varepsilon \to 0\) in the last inequality, and using the continuity of  \(\mathfrak{c}\), \(\mathfrak{B}\), and \(\varrho\), and recalling the convergences in  \eqref{CP1},  \eqref{CP2}, and invoking \eqref{Eq3.9}, we obtain
\[
\mathfrak{f}_1(x_0) - \mathfrak{f}_2(x_0) \leq \left[ \min_{\overline{\Omega}} \mathfrak{c} \right] (\mathscr{F}(\mathfrak{u}(x_0)) - \mathscr{F}(\mathfrak{v}(x_0))),
\]
which contradicts the assumptions $\textbf{(a)}$ and $\textbf{(b)}$. Therefore, we conclude that $\mathfrak{u} \geq \mathfrak{v}$ in $\Omega$.
\end{proof}

Next, as a consequence of the Comparison Principle, Theorem \ref{CPG}, we establish the existence of viscosity solutions to the Dirichlet problem:
\begin{equation}\label{DirichletProblem}
\left\{
\begin{array}{rclcl}
    \mathcal{G}(x, D \mathfrak{u}, D^2 \mathfrak{u}) & = & f(x) & \text{in} & \Omega  \\
    u(x) & = & g(x) & \text{on} & \partial \Omega. 
\end{array}
\right.
\end{equation}

\begin{theorem}[\bf Existence of Solutions]\label{Exist_Uniq_sol}
Suppose the assumptions of Theorem \ref{CPG} are satisfied. Additionally, assume that either $\displaystyle \inf_{\Omega} f(x) > 0$ or $\displaystyle \sup_{\Omega} f(x) < 0$. Suppose that there exists a viscosity subsolution $u_{\flat} \in \mathrm{C}^0(\Omega) \cap \mathrm{C}_{\text{loc}}^{0,1}(\Omega)$ and a viscosity supersolution $u^{\sharp} \in \mathrm{C}^0(\Omega) \cap \mathrm{C}_{\text{loc}}^{0,1}(\Omega)$ to the equation \eqref{DirichletProblem} such that $u_{\flat} = u^{\sharp} = g \in \mathrm{C}^0(\partial \Omega)$. Define the class of functions
\begin{equation}
\mathcal{S}_g(\Omega) := \left\{ v \in \mathrm{C}^0(\Omega) \mid v \text{ is a viscosity supersolution to } \eqref{DirichletProblem}, \right. 
\end{equation}
\begin{equation*}
    \left. \text{such that } u_{\flat} \leq v \leq u^{\sharp} \text{ and } v = g \text{ on } \partial \Omega \right\}.
\end{equation*}
Then, the function
\begin{equation}
    u(x) := \inf_{\mathcal{S}_g(\Omega)} v(x), \quad \text{for } x \in \Omega,
\end{equation}
is a continuous (up to the boundary) viscosity solution to the problem \eqref{DirichletProblem}.
\end{theorem}

\begin{example}[{\bf Non-Uniqueness of solutions - \cite[Example 1.3]{BessaDaSilva2025}}]
In general, the uniqueness of solutions to the Dirichlet problem  
\[
\left\{
    \begin{array}{rclcl}
        \Psi(x,|D u|)\Delta_{p}^{\mathrm{N}}u+\mathscr{H}(x, D u) & = & f(x) & \text{in } & \Omega, \\
        u(x) & = & g(x) & \text{on }& \partial \Omega
    \end{array}
\right.
\]  
is not guaranteed if the right-hand side does not preserve a fixed sign.  

As a concrete illustration, let us assume $i_\Psi = s_\Psi = \theta$ and $\mathscr{H}(x,\xi) = \langle \mathfrak{B}(x), \xi \rangle + \varrho(x)|\xi|^{\sigma}$. Then, both the trivial function $u(x) = 0$ and the nontrivial profile $v(x) = \mathrm{c}(1 - |x|^{\beta})$, with $\beta = \tfrac{2+\theta}{1+\theta}$ and $\mathrm{c} = \beta^{-1}$, turn out to be viscosity solutions of  
\[
\left\{
\begin{array}{rclcl}
|D u(x)|^{\theta} \left( \Delta_p^{\mathrm{N}} u(x) + \langle \mathfrak{B}(x), D u \rangle \right) + \varrho(x) |D u(x)|^{\sigma} & = & 0 & \text{in } & B_{1}, \\
w(x) & = & 0 & \text{on }  & \partial B_{1},
\end{array}
\right.
\]  
where $\mathfrak{B}(x) = \tfrac{p-2}{1+\theta}|x|x$, $\varrho(x) = \tfrac{n+(n-1)\theta}{1+\theta}|x|^{1+\theta-\sigma}$, and the parameters satisfy $\theta < \sigma < 1+\theta$.
\end{example}
\medskip
\section{The Strong Maximum Principle and Liouville-type result}
In this section, using the Comparison Principle,  we first address the proof of the Strong Maximum Principle below:
\begin{proof}[{\bf Proof of Theorem \ref{thm:SMP}:}]
    Let us assume that \( v \) is not identically zero in \( \Omega \). We aim to show that \( v > 0 \) in \( \Omega \). Supposing the  contrary that there exist \(x_0 \in \Omega\) and \(r > 0\) such that \(\mathrm{B}_{2r}(x_0) \Subset \Omega \), \( v > 0 \) in \(\mathrm{B}_r(x_0) \), and \( v(z) = 0 \) for some \( z \in \partial \mathrm{B}_r(x_0) \). Without loss of generality, we assume \( x_0 = 0 \).

Define the function \( u(x) = e^{-\alpha |x|^2} - e^{-\alpha r^2} \), with \(\alpha > 0\) to be chosen \textit{a posteriori}. Clearly, \( u > 0 \) in \( \mathrm{B}_r \) and \( u = 0 \) on \(\partial \mathrm{B}_{r}\). Consider the differential operator
\[
\mathscr{L}[u](x)=\Psi(x,|Du(x)|)\,\Delta_{p}^{\mathrm{N}}u(x) + \mathscr{H}(x, D u)  + \mathfrak{c}(x)\,u^{1+\sigma}(x).
\]
A direct computation yields that for \( \frac{r}{2} \leq |x| \leq r \) and for sufficiently large $\alpha >1$ ({to be chosen {\it a posteriori}}), we have $|Du(x)|=|-2\alpha x e^{-\alpha |x|^2}|<1$. Therefore, by Remark \ref{obsdePsi}, $\Psi(x,|Du|)\geq {\mathfrak{a}}{\mathfrak{L}_2}|Du|^{s_\Psi}$. Thus, for sufficiently large $\alpha>1$, the following inequality holds
 for \(s_\Psi<\sigma<1+s_\Psi\):
\begin{eqnarray}\label{coddebarreirae2}
    \mathscr{L}[u](x) &\geq& {\mathfrak{a}}{\mathfrak{L}_2}|Du(x)|^{s_\Psi}\Delta_{p}^{\mathrm{N}}u(x)-\varrho(x)|Du(x)|^\sigma-\|\mathfrak{c}\|_{L^\infty(B_{2r})}u^{1+\sigma}(x)\nonumber\\
    &=& (2\alpha e^{-\alpha|x|^2}|x|)^{1+s_\Psi} \Bigg[ {\mathfrak{a}}{\mathfrak{L}_2} 2|x|\alpha(p-1) - \frac{(n+p-2)}{|x|} \nonumber\\
    &&\qquad- \varrho(x)\left(2\alpha|x|e^{-\alpha|x|^2}\right)^{\sigma-1-s_\Psi}\nonumber\\&&\qquad-\|\mathfrak{c}\|_{L^\infty(B_{2r})}\frac{e^{-\alpha|x|^2\sigma}}{\alpha\left(2\alpha|x|e^{-\alpha|x|^2}\right)^{s_\Psi}} \frac{\left(1-e^{-\alpha(r^2-|x|^2)}\right)^{1+\sigma}}{2|x|}\Bigg] \nonumber\\
    &\geq& \mathrm{C}_*\cdot (2\alpha e^{-\alpha|x|^2}|x|)^{1+s_\Psi} \Bigg[ \alpha(p-1)r - \frac{2(n+p-2)}{r}- \|\varrho \|_{L^{\infty}(B_{2r})} \nonumber\\ 
    &&\qquad - \| \mathfrak{c} \|_{L^{\infty}(B_{2r})} \frac{1}{r}\left(1 - e^{-\alpha (r^2-|x|^2)} \right)^{1+\sigma} \Bigg] > 0,  
    \end{eqnarray}
provided $\alpha=\alpha(r, n, p, \|\mathfrak{c}\|_{L^\infty(\Omega)}, \|\varrho\|_{L^\infty(\Omega)}, \mathrm{C}_*, \sigma, s_\Psi)>1$ is chosen sufficiently large, where $\mathrm{C}_*>0$ is a constant which may depend on $\mathfrak a, \mathfrak L_2$.
Since \( u > 0 \) on \( \partial \mathrm{B}_{r/2} \), we can choose \( \ell_* \ll 1 \)  ({to be specified later}) such that \( \ell_* u \leq v \) on \( \partial \mathrm{B}_{r/2} \). Now using the Comparison Principle (i.e., Theorem \ref{CPG}), we have \(v \geq \ell_* u\) in the annular region \(\overline{\mathrm{B}}_{r} \setminus \mathrm{B}_{r/2}\).

{Moreover, \( \ell_* u \leq 0 \leq v\) in \( \mathrm{B}_{2r} \setminus \mathrm{B}_r \), so \( \ell_* u \) touches \( v \) from below at the point \( z \) and it is given that $v(z)=0.$} Hence, recalling the definition of viscosity solutions  with the test function \(\ell_* u\), we get
\begin{align*}
&\mathscr{L}[\ell_*u](z) = \Psi(z,|D(\ell_*u(z))|)\Delta_{p}^{\mathrm{N}} (\ell_* u(z)) + \mathscr{H}(z,D (\ell_*u(z)) )+ \mathfrak{c}(z)\, (\ell_* u(z))^{1+\sigma}\\
&\leq\Psi(z,|D(\ell_*u(z))|)\left[ -2\alpha \ell_* e^{-\alpha |x|^2} \left( n - 2\alpha |x|^2 \right) + (p - 2)\left( 1 - 2\alpha |x|^2 \right)\right]\\
&\qquad+ \|\varrho\|_{L^\infty(\Omega)} 2\alpha \ell_* e ^{-\alpha|x|^2} |x|\\
&\leq 0
\end{align*} for sufficiently large $\alpha\gg1$ (dependent on previous parameters as mentioned above) and sufficiently small $\ell_*=\ell_*(n, p, \alpha,r, \|\varrho\|_{L^\infty(\Omega)})\ll1,$
which  contradicts \eqref{coddebarreirae2}. Therefore, \( v > 0 \) in \( \Omega \).

To conclude the proof, we reiterate the above argument and use the fact that for any \( t > 0 \), there exists a constant \( \mathrm{C}_t > 0 \) such that \( 1 - e^{-s} \geq \mathrm{C}_t s \) for all \( s \in [0,t] \), thus, the proof is complete.
\end{proof}

\begin{remark}
It is worth highlighting that  the Proposition \ref{thm:SMP} holds true for the following cases:
\begin{itemize}
    \item[\checkmark] If we replace the term $\mathfrak c(x) v^{1+\sigma}$ by  $\mathfrak c(x) v^{1+s_\Psi}$.
    \item[\checkmark] If we consider $\mathscr{H}(x,\xi)=\langle\mathfrak{B}(x), \xi\rangle |\xi|^{s_\Psi}+ \varrho(x)|\xi|^{\sigma},$ where $s_\Psi< \sigma< 1+s_\Psi$, $\mathfrak B \in C^0(\overline \Omega; \mathbb R^n),$ and $0\leq\varrho \in C^0
(\Omega)$.
\item[\checkmark] If we consider general $F_0$ (see \eqref{g}) in place of $\Delta_p^{\mathrm{N}}$.
\item [\checkmark] If we consider any combination of the above three points separately.
\end{itemize}
\end{remark}

Now, to conclude this section, we establish a Liouville-type result.

\begin{proof}[{\bf Proof of Theorem \ref{LT} :}]
First, we claim that  there exists a compact set $\mathcal{K}\subset \mathbb{R}^n$ such that
\begin{align}\label{clm}\inf_{\mathbb R^n
}u = \min_{\mathcal{K}} u. 
\end{align} To achieve this, first, without loss of generality, by translating, we assume that $u\geq1$ in $\mathbb{R}^n$. By the hypothesis \eqref{hp}, let $\mathcal{K} \subset \mathbb{R}^n$ be a compact set such that $0\in \mathcal{K}$  and $(\mathfrak{B}(x)\cdot x)_+< n$ for all $x\in \mathcal{K}^c.$ Now define $\eta_\alpha(x):=|x|^\alpha,$ $\alpha\in (-1,0).$ For $x\in \mathcal{K}^c$ with $|x|>1,$ $|D\eta_\alpha|=|\alpha| |x|^{\alpha-1}<1.$ Then, we deduce the following:
\begin{align}\label{fv}
 &\Psi(x,|D\eta_\alpha(x)|)\,\Delta_{p}^{\mathrm{N}}\eta_\alpha(x) +   \langle\mathfrak{B}(x), D\eta_\alpha\rangle \Psi(x,|D  \eta_\alpha|)+ \varrho(x)|D \eta_\alpha|^{\sigma}\nonumber\\
 &\geq\Psi(x,|D\eta_\alpha(x)|)[\alpha\{n+(\alpha-1)(p-2)+\alpha-2\}|x|^{\alpha-2}]\nonumber\\
&\qquad\qquad+\Psi(x,|D\eta_\alpha(x)|)\alpha|x|^{\alpha-2}\left(\mathfrak{B}(x)\cdot x\right)\nonumber\\
 &=\Psi(x,|D\eta_\alpha(x)|)\alpha |x|^{\alpha-2}\Big[  (\alpha-1)(p-1)+\left( \left(\mathfrak{B}(x)\cdot x\right)_+- n\right)\nonumber\\ &\qquad\qquad\qquad\qquad- \left( \mathfrak{B}(x)\cdot x\right)_- \Big]\nonumber\\&>0.
\end{align}
Therefore, following \eqref{fv}, we can deduce  that $[\displaystyle\min_{\mathcal{K}} u] \Theta_\alpha \eta_\alpha$ is also a viscosity subsolution to \eqref{me} in $B_{\mathrm R}\setminus \mathcal K$ for all large $\mathrm{R}>0$,  where $\displaystyle \Theta_\alpha :=[\max_{\partial \mathcal{K}} \eta_\alpha]^{-1}$. Hence,  by the Comparison Principle (Theorem \ref{CPG}), we have $\displaystyle u \geq [\min_{\mathcal{K}} u] \Theta_\alpha \eta_\alpha$ in $B_{\mathrm{R}} \setminus \mathcal{K}$. 
Thus, letting $\mathrm{R} \to +\infty$, we get 
$$
\displaystyle u(x) \geq [\min_{\mathcal{K}} u]\Theta_\alpha \eta_\alpha(x) \quad \text{for all} \,\,\,x \in \mathcal{K}^c\,\,\,\text{
and for} \,\,\,-1<\alpha<0.
$$
Now, passing to the limit
$\alpha \to 0^{-}$ in the last relation, we get  \eqref{clm}.\\

Finally, using this, define the function \[w(x) =
u(x) - \min_{\mathcal{K}} u.\]  Clearly $w$ is  non-negative and satisfies the following: \[\Psi(x,|Dw(x)|)\,\Delta_{p}^{\mathrm{N}}w(x) + \mathscr{H}(x, D w)  \leq 0 \quad \text{in}    \quad \mathbb R^n.\] Moreover, $w$ vanishes somewhere in $\mathcal{K}$. Therefore, by the Strong Maximal Principle (Theorem \ref{thm:SMP}), $w\equiv 0$, which implies that $u$ is constant. This completes the proof.
\end{proof}

\subsection*{Acknowledgments}

FAPESP-Brazil has supported J. da Silva Bessa under Grant No. 2023/18447-3. R. Biswas is supported by FONDECYT-Chile postdoctoral Project no. 3230657, J.V. da Silva has received partial support from CNPq-Brazil under Grant No. 307131/2022-0 and FAEPEX-UNICAMP 2441/23 Editais Especiais - PIND - Projetos Individuais (03/2023). G.S. S\'{a} expresses gratitude to the Centro de Modelamiento Matemático (CMM), BASAL fund FB210005, for the center of excellence from ANID-Chile. M. Santos is partially supported by the Portuguese government through FCT-Funda\c c\~ao para a Ci\^encia e a Tecnologia, I.P., under the projects UID/04561/2025 and the Scientific Employment Stimulus - Individual Call (CEEC Individual), DOI identifier https://doi.org/10.54499/2023.08772.CEECIND/CP2831/CT0003.






\end{document}